\documentclass[12pt,reqno]{amsart}
\usepackage{amsmath}
\usepackage{amsthm}
\usepackage{amscd}
\usepackage{amsfonts}
\usepackage{amssymb}
\usepackage{mathrsfs}
\usepackage{graphicx}
 \usepackage[usenames]{color}
\usepackage[top=1.25in, bottom=1.25in, left=1.25in, right=1.25in]{geometry}
\usepackage{natbib}

\definecolor{red}{rgb}{1.0,0.0,0.0}

\definecolor{blu}{rgb}{0.0,0.0,1.0}
 \definecolor{gre}{rgb}{0.0,0.5,0.2}
\def\gre#1{{\textcolor{gre}{#1}}}

\newtheoremstyle{mytheorem}
{6pt}
{6pt}
{\itshape}
{-0pt}
{\large \scshape}
{}
{1em}
{}

\newtheoremstyle{myremark}
{6pt}
{10pt}
{\rm}
{-0pt}
{\large \scshape}
{}
{1em}
{}

\theoremstyle{mytheorem}
\newtheorem{Theorem}{Theorem}
\newtheorem{Definition}{Definition}
\newtheorem{Proposition}{Proposition}
\newtheorem{Lemma}{Lemma}
\newtheorem{Corollary}{Corollary}
\newtheorem*{Hypothesis}{Hypothesis}

\theoremstyle{myremark}
\newtheorem{Remark}{Remark}
\newtheorem{Example}{Example}

\newlength{\bibitemsep}
\newlength{\bibparskip}
\let\oldthebibliography\thebibliography
\renewcommand\thebibliography[1]{
	\oldthebibliography{#1}
	\setlength{\parskip}{\bibitemsep}
	\setlength{\itemsep}{\bibparskip}
}

\bibpunct{(}{)}{; }{a}{,}{,}
\selectfont

\begin{document}
	\title[Centralized and Competitive Extraction for Distributed Renewable Resources]{Centralized and Competitive Extraction for Distributed Renewable Resources with Nonlinear Reproduction} 
	\author[F.De Feo]{Filippo De Feo$^{**}$}
    \address{Institut f\"ur Mathematik, Technische Universit\"at Berlin, Berlin, Germany}
    \email{defeo@math.tu-berlin.de}
    \author[G.Fabbri]{Giorgio Fabbri$^*$}
	\address{$^*$Univ. Grenoble Alpes, CNRS, INRAE, Grenoble INP, GAEL, 38000 Grenoble, France}
	\email{giorgio.fabbri@univ-grenoble-alpes.fr}
	\author[S.Faggian]{Silvia Faggian$^\dag$}
	\address{$^\dag$Department of Economics, Ca' Foscari University of Venice,
		Italy}
	\email{faggian@unive.it}
	\author[G. Freni]{Giuseppe Freni$^\ddag$}
	\address{$^\ddag$Department of Business and Economics, Parthenope University
		of Naples, Italy}
	\email{giuseppe.freni@uniparthenope.it}
	
	\begin{abstract}

We study optimal and strategic extraction of a renewable resource that is distributed over a network, migrates mass-conservatively across nodes, and evolves under nonlinear (concave) growth. A subset of nodes hosts extractors while the remaining nodes serve as reserves. We analyze a centralized planner and a non-cooperative game with stationary Markov strategies. The migration operator transports shadow values along the network so that Perron–Frobenius geometry governs long-run spatial allocations, while nonlinear growth couples aggregate biomass with its spatial distribution and bounds global dynamics. For three canonical growth families, logistic, power, and log-type saturating laws, under related utilities, we derive closed-form value functions and feedback rules for the planner and construct a symmetric Markov equilibrium on strongly connected networks. To our knowledge, this is the first paper to obtain explicit policies for spatial resource extraction with nonlinear growth and, a fortiori, closed-form Markov equilibria, on general networks.

\bigskip
		
		 
	\end{abstract}
	
	\maketitle

\section{Introduction}






Many renewable resources are distributed across space and move between locations, so that extraction at one site affects stocks and payoffs elsewhere. Classical bioeconomic models describe aggregate dynamics for a single homogeneous stock (for example, \citealp{Gordon1954,Schaefer1954}), while a more recent large spatial literature shows how migration and connectivity across patches shape both efficiency and dynamics. This is the concept of \textit{metapopulations}: multiple sub-populations linked by dispersal, highlighting how patch connectivity and migration affect species persistence, see \cite{Hanski1999}.

The most standard example of space-distributed natural resource is fish, but other empirically relevant examples are also groundwater stored in hydraulically connected aquifers, shared oil and gas reservoirs with pressure–driven flows, forest resources with seed dispersal and pest spread, and even ``non–natural'' stocks such as pollutants that drift and diffuse from their sources. In all these cases, the global mass of the resource and its spatial allocation jointly determine marginal productivities, incentives, and long–run outcomes, generating spatial externalities that standard non–spatial models cannot capture \citep[e.g.][]{janmaat2005sharing}.

The idea of metapopulation (or spatially distributed resource) was first used in the context of optimal spatial resource management by \cite{SCW05} (see also \citealp{BhatHuffaker07, SSW09, XE10, BXEY14} and the stochastic counterpart by \citealp{CostPol08}).
These first contributions did not model strategic interaction, but questions about the ``tragedy of the commons'' in metapopulations or spatial resource games, and the related policy implications, quickly arose: \cite{KaffineCostello11} provide an early game-theoretic model of a metapopulation fishery with territorial use rights for fishing, while \cite{CTQ15} build on this with the concept of ``partial enclosure.'' Extensions and related contributionsinclude, for instance,   \cite{HerreraEtAl2016, CostelloKaffine18, CostelloEtAl2019, de2019spatial}\footnote{It is also worth mentioning \cite{Ilkilic2011}, who studies equilibrium extraction on a network in the static case.}.

Linear or linearized stock dynamics are analytically convenient in spatial control/games problems because they often permit closed-form (or nearly closed-form) characterizations of optimal policies when aggregation is tractable, yet many renewable stocks exhibit density dependence and environmental constraints that are inherently nonlinear. For instance: concave growth   (surplus–production)   functions can capture both diminishing marginal biological productivity at higher stock levels as in the Schaefer surplus–production framework and its descendants \citep{Schaefer1957,HilbornWalters1992,Clark2010}, and carrying–capacity constraints can reproduce saturation effects that pin down the long–run stock around a capacity level (as in \cite{HilbornWalters1992,Clark2010});  threshold phenomena can arise, for instance at low densities, with recruitment or survival falling disproportionately (Allee effects), so generating tipping behaviour, multiple steady states, and extinction traps \citep{CourchampBerecGascoigne2008}. 

In networked models with linear stock dynamics (linear growth and linear migration), Perron–Frobenius-like properties of the involved matrices yield spectral solutions: in \citet{FFF24} (see also \citealp{FFF20}) both the planner and the Markovian game admit affine feedbacks proportional to weighted aggregate stock \footnote{The weights are given by the dominant right eigenvector of a matrix combining the regional productivities and the migration flows.}, and the long-run spatial allocation is governed by the left dominant eigenvector of the matrix obtained subtracting the feedbacks coefficients from the flow matrix. Related linear/linearized network formulations obtain equally transparent solutions or sharp comparative statics by leveraging the same spectral structure \citep{calvia2024optimal, FFF25}\footnote{Other related contributions include dynamic optimization or differential game models for other economic systems with spatial structure, for example, growth models with distributed capital that are counterparts of benchmark models such as Ramsey or AK, either in continuous space (\citealp{Brito04, BCF13, fabbri2016geographical, SXY17, boucekkine2019growth, brito2022dynamics, ricci2024non}) or on networks (\citealp{albeverio2022large}), as well as dynamic economic models with negative spatial externalities related to pollution diffusion (\citealp{de2019spatial, de2022investment, boucekkine2022managing, boucekkine2022dynamic}). The explicit result of this paper is a novelty even when compared with this literature.}. 


Here, we develop a dynamic model of a spatially distributed renewable resource on a network with mass–conserving migration and, notably, with a \emph{nonlinear} (concave) growth. A subset of nodes hosts extractors who choose paths over time; the remaining nodes act as reserves. We analyze $(i)$ a centralized benchmark in which a planner maximizes discounted welfare (Section \ref{sec:SC}) and $(ii)$ a noncooperative game with stationary Markovian strategies (Section \ref{sec:game}). In both settings, the migration operator transports shadow values across the network, so Perron–Frobenius geometry continues to shape long–run spatial allocations; nonlinearity in growth, however, bounds aggregate dynamics and couples the global mass with the distribution of stocks across nodes. For each of three standard concave growth families (logistic/power/log–type saturating laws) \footnote{For these specifications in the case of a non-distributed resource see \cite{PY1989} and \cite{GL2007}.}, we obtain closed–form value functions and feedback rules for the planner (Theorem \ref{thm:OC}) and derive a symmetric Markovian equilibrium in the game (Theorem \ref{th:ME}). {To the best of our knowledge}, this is the first paper to deliver an explicit solution for optimal extraction of a spatially distributed resource {with nonlinear growth}, and, \emph{a fortiori}, the first to provide a closed–form Markov equilibrium under the same assumptions on a general (strongly connected) network. That said, our explicit formulas hinge on an aggregate growth function, on pairing, in each of the three cases, the nonlinear growth law with the corresponding utility (CRRA or log), on a strongly connected linear mass–conserving network and on symmetric extractors at the active nodes; we take this as a disciplined benchmark for subsequent extensions.

The remainder of the paper is organized as follows.
Section \ref{sec:model} introduces the model setup, including the network structure, resource dynamics, and agents' optimization problems.
Section \ref{sec:SC} derives the socially optimal solution, presenting the planner's HJB equation and the closed form optimal policy.
Section \ref{sec:game} analyzes the noncooperative game, characterizes the stationary Markov perfect equilibrium, and provides an explicit formula for the equilibrium extraction strategy at each node.
Section \ref{sec:conclusions} concludes with a summary of findings and directions for future research.
Appendix \ref{app:proofs} collects technical lemmas and proofs.

\section{The Model}  	\label{sec:model}

We consider a common property resource distributed across an area divided into subareas or regions. The resource's local quantities change over time and are \emph{mobile} between regions, in specified proportions. This model can represent various types of resources, such as livestock migrating across regional or national boundaries, or groundwater or oil supplies. For illustration, one can focus on fish that move between different regional waters. The overall area  is modeled generically as  a network , with $n$ nodes, each  representing a different region,   connected to one another by suitable inflow/outflow functions. A number $f\le n$ of agents  are assigned each to a single region for exclusive exploitation and decide the quantity of resource they extract in time.
 \subsubsection*{State and control variables.} 
Specifically, we denote by   $N=\{1, 2, ..., n\}$ the set of nodes, $F\subset N$ the subset where an agent is active, with the remaining $N\setminus F$ nodes used as reserves. Here $e_{i}$ is the $i$-th vector of the canonical basis on $\mathbb{R}^{n}$,   $\mathbf e =\sum_{i=1}^n e_i$ the unit vector,   $\left\langle \cdot ,\cdot \right\rangle $ is the inner product  in $\mathbb{R}^{n}$, and   $\mathbb{R}_+$ the set of nonnegative real values. For any $x\in\mathbb R^n$, we write $x\ge 0$ to mean $x_i\ge0$ for all $i\in N$, and $x>0$ to mean $x_i>0$ for all $i\in N$.

	For all $i\in N$, $X_{i}(t) $ stands for the mass at 	node $i$ at time $t$, and we set $X(t)=(X_{1}(t),\dots,X_{n}(t) )^\top$. The evolution in time of mass $X_{i}(t)$ on region $i$ depends on several factors:	
  the natural growth  $g_i(X(t))$ of the resource at node $i$;  the net inflow $\alpha_i(X(t))$ of the resource stock exchanged with the other nodes;   the rates of extraction $c_{i}(t)$ at time $t$ from region $i$, namely the decision variables of the problem,  chosen by the agents. 

The evolution of the system is then described by the set of ODEs 
\begin{equation}  \label{SE}
\begin{cases}
X^\prime_{i}(t)=g_{i}(X(t))+\alpha_i(X(t))-c_{i}(t), & 
t>0 \\ 
X_{i}(0)=x_{i}, & 
\end{cases}%
\end{equation}
subject to the nonnegativity constraints $c_{i}(t)\geq 0$, $X_{i}(t)\geq 0$,  for all $t\geq 0$.

We denote by $X(t;x\mathbf{,}c\left( \cdot \right) )$ the trajectory of the system starting at $x$ and driven by control $c(t)=(c_1(t), c_2(t), \cdots, c_n(t))^\top$.

We now specify further $g_i$ and $\alpha_i$. For the net inflow $\alpha_i$, we mainly consider the case of a   diffusion linear in the regional stocks, more specifically we make the following assumption.

\begin{Hypothesis}[N1] The stock exchanges between two nodes of the network are represented by the network's adjacency matrix $B$, with entries $b_{ij}\ge0$ for all $i,j$ in $N$, and  $b_{ii}=0$. 
The outflow from node $i$ to  node $j$ at time $t$ is represented by $b_{ij}X_{i}(t)$. \\The network  is assumed \emph{strongly connected}.    \end{Hypothesis}

Under this assumption one has 

$$\alpha_i(X(t))= \sum_{j=1,j\not=i}^{n}b_{ji}X_{j}(t)  - 
\sum_{j=1,j\not=i}^{n}b_{ij}X_{i}(t) =\left\langle Be_{i},X(t)\right\rangle
- \sum_{j=1,j\not=i}^{n}b_{ij}\; X_{i}(t),$$
If we set  
$\alpha(x)=(\alpha_1(x), \alpha_2(x), \cdots, \alpha_n(x))^\top$,  and $D$  the diagonal matrix of the  outflows
$$D=\begin{pmatrix} -\sum_{j=1,j\not=1}^{n}b_{1j}&0& 0&\cdots &0\\
0&-\sum_{j=1,j\not=2}^{n}b_{2j}&0&  \cdots&0\\
\cdots&\cdots&\cdots&  \cdots& \cdots&\\
0&0&0&  \cdots& -\sum_{j=1,j\not=n}^{n}b_{1j}&\\
\end{pmatrix} $$
then 
$$\alpha(x)= (D+B^\top)x, \ x\in\mathbb R^n.$$ 
Note that, under Hypothesis (N1) one has
\begin{equation}\label{eq:consmass}
\sum_{i=1}^n \alpha_i(X(t))=\langle \alpha (X(t)),\mathbf{e}\rangle=0,\quad \forall t\ge0.
\end{equation} 
meaning that the total mass of the reserve is conserved when moved across nodes.

\begin{Example}[Symmetric diffusion of Fick's type]\label{ex:Fick}
A useful special case of Hypothesis (N1) is a symmetric diffusive coupling in the spirit of Fick’s law. In this setting, the flow between two nodes is linear and proportional to the difference in their stocks.
Equivalently, let $W=(w_{ij})_{i,j\in N}$ be a nonnegative weight matrix with $w_{ij}=w_{ji}$ for $i\ne j$ and $w_{ii}=0$. 
Interpreting $w_{ij}$ as the intensity of exchange along edge $(i,j)$, the net inflow at node $i$ reads
\[
\alpha_i(x)\;=\;\sum_{j\ne i} w_{ji} (x_j-x_i).
\]
\end{Example}

\begin{Remark}\label{rem1}
      We will discuss a generalization of Hypothesis (N1), for possibly non-linear net flows $\alpha_i$, in Remark \ref{rem:nonlinearalpha}.
  \end{Remark}

\noindent As for the  \emph{growth function }$g_i$, we  assume  $g_i$ is  a real valued concave function  of one of the following types.

\begin{Hypothesis} For all $i\in N$
    $$
g_{i}(x)=x_{i}\ \varphi \left( \sum_{i=1}^{n}x_{i}\right)=x_{i}\ \varphi \left( \langle x,\mathbf e\rangle\right)
$$
where the \emph{saturation function} $\varphi $ can be chosen among the following: \begin{itemize}
\item[(S1)] $\varphi \left( m\right) =\Gamma \left( 1-\frac{1}{K}m^{\sigma
-1}\right) ,$  with $\sigma >1$, $\Gamma>0$ usato nel thm di verifica;

\item[(S2)] $\varphi \left( m\right) =m^{\sigma -1}-\delta,$   with $0<\sigma <1,$
and $\delta >0;$

\item[(S3)] $\varphi \left(m\right) =\Gamma \left( 1-\frac{1}{K}\ln m\right) ,$   with 
$\Gamma ,K\in \mathbb{R}.$
\end{itemize}
where $m$ varies in $(0,+\infty) $,  and $\Gamma,\ K,\ \delta,\ \sigma$ are given parameters, representing:  $\Gamma $  the growth rate of the resource,  $K$    the carrying capacity of the environment, $\delta$  the natural decay of the resource.
\end{Hypothesis} 

Overall, if we set $c(t)=(c_1(t), c_2(t), \cdots, c_n(t))^\top$, the vector of extractions chosen by agents,  system \eqref{SE} can be rewritten in vector form as
\begin{equation}\label{vectsys}\begin{cases}
X^\prime(t)=\varphi \left( \langle X(t),\mathbf e\rangle\right)X(t)+(D+B^\top)X(t)-c(t), &t\ge 0\\
X(0)=x.
\end{cases} 
\end{equation} 
  subject to the constraints 
\begin{equation}  \label{constr}
c_{i}(t)\geq 0,\,\ \ X_{i}(t)\geq 0,\text{ for all }t\geq 0.
\end{equation} 

We finally describe the agents and their utility. 

\begin{Hypothesis}
    Assume 
$\rho \in \mathbb{R}$ is the discount rate. We suppose that at each location $i\in F$ there is an agent with the following instantaneous utility:
\begin{itemize}
\item[(U1)] $\displaystyle u(c)=\frac{c^{1-\sigma}}{1-\sigma}$ \quad when the saturation function $\varphi$ satisfies (S1) or (S2);
\item[(U2)] $\displaystyle u(c)=\ln(c)$ \quad when the saturation function $\varphi$ satisfies (S3).
\end{itemize}
\end{Hypothesis}

We will analyse and compare two frameworks: 
\begin{itemize}
    \item [(F1)] In Section \ref{sec:SC} we will study the case of a unique ``Benthamite'' planner maximizing the sum of the utilities from the extraction of the resource at the different nodes 
\begin{equation}\label{utcontr}
J\left(c; x\right)=\int_{0}^{+\infty }e^{-\rho t}\sum_{i=1}^fu(c_{i}(t))dt  
\end{equation}
under the constraints \eqref{vectsys} and \eqref{constr};
\item[(F2)] In Section \ref{sec:game} we will look at the case of $f$ competing agents in a dynamic game, with agent $i$  maximizing with respect to $c_i$  \begin{equation}
\label{utility}
J_i\left(c_i ; x, c_{-i}\right)=\int_{0}^{+\infty }e^{-\rho t}u(c_{i}(t))dt  
\end{equation}
under \eqref{vectsys} and \eqref{constr}.

Note that in \eqref{utcontr} and \eqref{utility}, $x$ is the initial stock   and the notation $J\left(c ; x, c_{-i}\right)$, (respectively, $J_i\left(c_i ; x, c_{-i}\right)$) highlights the fact that $J$ (respectively,  $J_i$) depends implicitly on $x$ (respectively,  $x,c_{-i}$)  through the constraint given by the state equation \eqref{vectsys}.

\end{itemize}

\begin{Remark}
Our specification $g_i(x)=x_i\,\varphi(\langle x,\mathbf e\rangle)$ imposes that congestion/limits act at the level of the \emph{aggregate} mass rather than node-by-node. While several spatial bioeconomic models use patch-level crowding, there are empirically motivated situations where a nonlocal limiter is a natural structure. One is, for instance,  benthic filter-feeders in shallow estuaries and bays, where phytoplankton is transported and mixed at the system scale in shallow, well-mixed estuaries \citep{Officer1982}. In other biological systems, aggregation at shared reproductive sites also generates congestion effects at the aggregate rather than local level. For instance, many salmon populations exhibit strong natal homing: adults from large catchments return to a few specific spawning reaches, where density dependence acts through redd superimposition and limited gravel availability. In such cases, reproduction is constrained by the system-wide carrying capacity of spawning grounds rather than by patch-specific crowding \citep{WestleyQuinnDittman2013}. 


More generally, in continuous-space models it is standard to include nonlocal terms whereby reproduction or mortality at a point depends on spatial averages or integrals of the state (kernel interactions) (\citealp{Murray2003,OkuboLevin2001}). Our dependence on $m:=\langle x,\mathbf e\rangle$ can be interpreted as a specific  nonlocal situation that captures basin- or region-wide crowding.
\end{Remark}

\subsection{Primitives of the system} It will prove convenient to represent the trajectory of the system $X(t)$ as the product \begin{equation}\label{ym}
     X(t)=m(t)Y(t)
 \end{equation}
 where $m(t)$ is  the   \emph{total mass}  of the resource, namely,  the sum  at time $t$ of the stocks at different nodes \begin{equation}\label{m}
    m(t)\equiv\sum_{i\in N}X_i(t)=\left\langle \mathbf{e} ,X(t)\right\rangle
\end{equation}
while $Y(t)$ is  the vector of the \emph{shares  of the total stock} in different regions, having the property
$$\sum_{i=1}^nY_i(t)=\langle \mathbf e, Y_i(t)\rangle\equiv 1, \quad \forall t\ge0,$$
that is, $Y(t)$ varies in time remaining on the simplex of equation
$$y_1+y_2+\cdots+y_n=1,\quad y_i\ge0.$$
hence $X(t)$ satisfies the equation in  \eqref{vectsys} if and only if $m(t)$ and $Y(t)$ satisfy
\begin{equation}\label{stateqgen}
    \begin{cases}
        m'(t)=\varphi(m(t))m(t)-\langle c(t),\mathbf e\rangle, &t\ge0\\
        Y'(t)=\left(D+B^\top+\frac{\langle c(t),\mathbf e\rangle}{m(t)}\right)Y(t)-\frac{c(t)}{m(t)}&t\ge0\\
        m(0)=\langle x, \mathbf e\rangle, \ Y(0)=\frac 1 {\langle x, \mathbf e\rangle}x
    \end{cases}
\end{equation}
\subsubsection*{Evolution of the system in absence of extraction}


We now examine the relationship between the system's evolution and the network parameters.

\begin{Lemma}    \label{lem:D+B}
Assume that Hypothesis $(N1)$ is satisfied. Then the matrix \( D+B\) has 0 as  eigenvalue, moreover all other eigenvalues have strictly negative real parts. More explicitly, if $\{0, \lambda_2, \cdots, \lambda_n\}$ are the eigenvalues of $D+B$, then  \begin{equation}
    \label{eq:eigen}    0>\emph{Re}\lambda_2\ge\emph{Re}\lambda_3\ge\cdots\ge\emph{Re}\lambda_n  \end{equation}
The   eigenspace associated to the zero eigenvalue has dimension 1, and is generated by  the     ({dominant}) eigenvector $\mathbf e$. Similarly, the transpose $D+B^\top$ has a  (dominant) eigenvector $\zeta$ associated to the zero eigenvalue, having all positive coordinates. All eigenvectors associated to nonzero eigenvalues have at least a nonpositive coordinate.
\end{Lemma} 

The reader will find the proof of this lemma in the appendix.

 \medskip

 The role of eigenvalues and eigenvectors is made immediately clear when analyzing the evolution of the system in absence of extraction. By differentiating the product in \eqref{ym},   using \eqref{eq:consmass}, and setting $c(t)\equiv 0$ in \eqref{vectsys}, one sees that  $m(t)$ and $Y(t)$ defined in \eqref{ym} \eqref{m} satisfy the equations
\begin{equation}\label{eq:mY}\begin{cases}
m ^{\prime }(t)=m (t)\varphi (m(t)), &t\ge 0\\
Y'(t)=(D+B^\top)Y(t), &t\ge 0\\
m(0)= \langle  \mathbf{e},x \rangle,\quad Y(0)=\langle \mathbf e,x\rangle^{-1}x\\
\end{cases} 
\end{equation}
and the following lemma holds.

\begin{Lemma}\label{lem:Y} 
Assume that Hypothesis $(N1)$ is satisfied.
Assume the initial stock    $x\ge0$, with $x$  nonzero, and assume  a null extraction $c(t)\equiv0$. Then  
\begin{align}\label{eq:Yformula}
    &Y(t)=\langle\mathbf e,x\rangle^{-1}e^{(D+B^\top)t}x, \quad \textrm{with} \quad Y_i(t)\ge0,\quad  \forall t\ge0, \forall i\in N, \\
    &\lim_{t\to\infty} Y(t)=\langle\mathbf{e} , \zeta\rangle^{-1} \zeta, \label{eq:Yformula2}
\end{align} where $\zeta$ is the dominant eigenvector defined in Lemma \ref{lem:D+B}. Moreover, $m(t)>0$ for all $t\ge0$ and:
\begin{itemize}
    \item[$(i)$] in assumptions $(S1)$, \ \ $m(t)=m_0(t):=\left(e^{-\Gamma(\sigma-1)t}\big[(\sum_ix_i)^{1-\sigma}-\frac 1 K\big]+\frac 1 K\right)^{\frac 1{1-\sigma}}$
    
    \item[$(ii)$] in assumptions $(S2)$,\ \ $m(t)=m_0(t):=\left(e^{-\delta(1-\sigma)t}\big[(\sum_ix_i)^{1-\sigma}-\frac 1 \delta\big]+\frac 1 \delta\right)^{\frac 1{1-\sigma}}$
    
    \item[$(iii)$] in assumptions  $(S3)$,\ \ 
    $\displaystyle m(t)=m_0(t):=\exp\left( e^{-\frac\Gamma K t}(\log(\langle x, \mathbf e\rangle)-K) +K\right) $.
\end{itemize}
So that, In the long run, as $t \to \infty$, we have
\[
\bar m := \lim_{t \to \infty} m(t) = 
\begin{cases}
K^{\frac{1}{\sigma-1}}, & \text{under assumptions}\  (S1) ,\\[2pt]
\delta^{\frac{1}{\sigma-1}}, & \text{under assumptions} \ (S2),\\[2pt]
e^{K}, & \text{under assumptions}\ (S3).
\end{cases}
\]
\end{Lemma}

The reader will find the proof of this lemma in the appendix, while the proof of the Corollary below is straightforward.
 
\begin{Corollary} \label{cor:X} In the assumptions of Lemma \ref{lem:Y}, and for $\bar m= \lim_{t\to\infty}m(t)$,     the trajectory of   system \eqref{vectsys}, with null extraction $c(t)\equiv0$, satisfies
$$ \bar x=\lim_{t\to\infty}X(t)=\bar m \langle \mathbf e, \zeta\rangle^{-1} \zeta$$
that is, it
converges to a point of the simplex of equation
$$x_1+x_2+\cdots+x_n=\bar m,\quad x_i\ge0.$$
\end{Corollary} 
 \begin{center}
    \includegraphics[width=8cm]{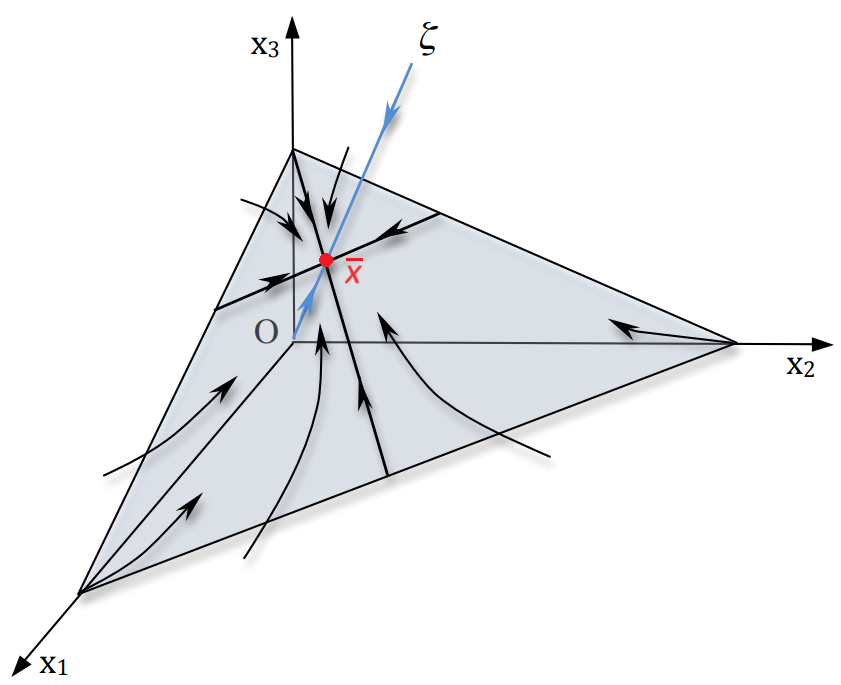} 
 \end{center}

Note that in the case of linear increments of the stock treated in \cite{FFF24}, the trajectory was either diverging or converging to zero along the direction of the dominant eigenvector. In the nonlinear cases treated here, the trajectory converges to a vector with all finite components, is amplified by the mass \( m(t) \), and it is distributed in the long run proportionally to the coordinates of the dominant eigenvector \(\zeta\).

That establishes also the relationship between the trajectory of the system and the network:  in the absence of extraction, the stock distributes proportionally to the node's eigenvector centrality of the matrix $D+B$.  Namely, it's the linear part of the system, embodied by $Y(t)$,  that dictates the proportions between the components of the limiting stock.

\section{The Social Planner Problem}\label{sec:SC}

We study a centralized benchmark in which a benevolent decision maker chooses the entire extraction vector $c(t)$ to maximize the discounted sum of agents' instantaneous utilities:
\[
\max_{c(\cdot)\ge 0}\; \int_0^{\infty} e^{-\rho t}\sum_{i\in F} u(c_i(t))\,dt \quad\text{subject to \eqref{vectsys}--\eqref{constr}.}
\]
In other words, we consider the problem of an utilitarian (Benthamite) social planner. It also coincides with the cooperative outcome under transferable utility and full commitment (in the \emph{grand coalition} case with uniform weights).

We address the problem by means of Bellman's Dynamic Programming. To this extent, if  $L_{loc}^{1}(0,+\infty ;\mathbb{R}^{n}_+)$ denotes the   set of non-negative vector valued and locally integrable functions on 
$[0,+\infty )$, and $X(t)=X(t;c,x)$ denotes the trajectory starting at $x$ and driven by a control $c$, we  set 
\begin{equation}\label{eq:admSC}
\mathbb A:=\{c\in L_{loc}^{1}(0,+\infty ;\mathbb{R}^{n}_+): c(t)\ge 0, X(t;c,x)\ge0, \forall t\ge0 \}    
\end{equation}
as the set of admissible strategies for the social planner. We will call any control $c\in\mathbb A$ and the associated trajectory an \textit{admissible couple}. 
Then the \emph{value function} for the social planner is defined as
$$
V(x)=\sup _{c \in \mathbb A\color{black}} J\left(c ; x\right).
$$
The Hamilton-Jacobi-Bellman (briefly, HJB) equation associated to the problem by Dynamic Programming is
\begin{equation} \label{hjbgen}
\rho v(x)=\max _{c \geq 0}\left\{\sum_{i=1}^f\left(u\left(c_i\right) - c_i \frac{\partial v}{\partial x_i}(x)\right)\right\}+ \langle   \nabla v(x), \alpha(x)+  \varphi(\langle x,\mathbf e\rangle)x\rangle
 \end{equation}  paired with the condition giving the maximizing control
\begin{equation}  \label{optcontgen}
c^*_i=\left(\frac{\partial v}{\partial x_i}(x)\right)^{-\frac1\sigma}, \forall i\in F, \quad c_i^*\equiv 0, \forall i\in N\setminus F.
\end{equation}
with $\sigma=1$ in the case of a logarithmic utility.  The reader will find in Section \ref{ssec:verif} in the Appendix how \eqref{hjbgen} can be explicitated in the different assumptions on the  data.

 Now consider the following set of parameters and quantities. 

 \begin{small} 
\begin{table} [h]
\begin{center}
\caption{ }

\label{TabOC}
\begin{tabular}{c|c|c|c|c|c  }
framework  &
$\varphi(m)$&
$u(c)$ &
$\theta^*$&
$A$&
$B$  
 \\
\hline\hline
 (S1)(U1)&
$\Gamma \left( 1-\frac{1}{K}m^{\sigma-1}\right) $&
$-\frac{c^{1-\sigma}}{\sigma-1}$, $\sigma>1$&
$\frac{\rho+\Gamma(\sigma-1)}{\sigma f}$&
${(\theta^*)^{-\sigma}} $&
$-\frac{\Gamma (\theta^*)^{-\sigma} }{K\rho}$ \\
\hline
 (S2)(U1) & 
$m^{\sigma -1}-\delta$&
$\frac{c^{1-\sigma}}{1-\sigma}$, $\sigma<1$ &
$\frac{\rho+\delta(1-\sigma)}{ \sigma f} $&
$ {(\theta^*)^{-\sigma}}$&
$\frac {(\theta^*)^{-\sigma}}{\rho}$ \\
\hline
 (S3)(U2)&
$\Gamma \left( 1-\frac{1}{K}\log m\right) $&
$\log c$ &
$\frac1f(\rho + \frac \Gamma K )$ &
$(\theta^*)^{-1}$&
$ \frac {\Gamma -f\theta^*+f\theta^*\log\theta^*}{\theta^*\rho }  $ 
 \end{tabular}
\end{center}
\end{table}
\end{small}

\begin{Theorem}[Existence and Uniqueness of the Optimal Strategy]\label{thm:OC}
In each of the assumption frameworks described in Table \ref{TabOC}, and assuming $\theta^*>0$, the following facts hold:
\begin{itemize}
    \item[$(i)$] when admissible,   the only  closed-loop  optimal  strategy $\psi^*$   is given by
\begin{equation}\label{eq:optccl}
 \psi_{i}^*=0, \ \  \forall i\not\in F \quad\textrm{and}  \quad \psi_{i }^*(x)=\theta^*\langle  x, \mathbf e\rangle,\ \forall i\in F;
\end{equation} 
 \item[$(ii)$]  in such case, the value function of the problem  is, 
\begin{equation}\label{eq:vfOC}
V(x)=A\;u(\langle\mathbf e,x\rangle)+B.
\end{equation} 
 \item[$(iii)$]  moreover, the value function $V(x)$ is the unique classical solution of the HJB equation \eqref{hjbgen}, in the sense of Theorems \ref{thm:ver1} and \ref{thm:ver2} in Appendix \ref{thm:OC}.
\end{itemize}
\end{Theorem} 
    
The proof of Theorem \ref{thm:OC} constitutes the contents of  Appendix \ref{ssec:verif}. Note that Theorems \ref{thm:ver1} and \ref{thm:ver2} and their proofs provide sufficient conditions for the value function to be the unique solution to the HJB equation (\ref{hjbgen}) under different set of assumptions.
It is also to be noted that such results rely on the proof of existence and uniqueness of the solution of the \emph{closed-loop equation}, as worked out in the next Section \ref{ssec:cle}.

\begin{Remark}
\label{rem:nonlinearalpha}
Let $\tilde\alpha:\mathbb{R}_+^n\!\to\!\mathbb{R}^n$ be locally Lipschitz and mass-preserving, i.e.
$\langle \mathbf e,\tilde\alpha(x)\rangle=0$ for all $x$ (and inward-pointing on $\partial\mathbb{R}_+^n$).
Replacing the linear migration $\alpha(x)=(D+B^\top)x$ by $\tilde\alpha(x)$ leaves $V(x)=A\,u(\langle \mathbf e,x\rangle)+B$ a good candidate  solution of the HJB equation in 
Theorem \ref{thm:OC}. Indeed in this case
$\nabla V(x)=A\,u'(\langle \mathbf e,x\rangle)\mathbf e$ so
$\langle \nabla V(x),\tilde\alpha(x)\rangle=0$ and the migration term drops out, hence the maximizer of the Hamiltonian coincides with $c^*$ in Theorem \ref{thm:OC}, and $V$
satisfies the same scalar identity in the aggregate $m=\langle \mathbf e,x\rangle$ as in the linear case. Conditions to characterize the admissibility of the candidate optimal trajectory and, later,  the candidate Markovian Nash equilibrium are harder, and the dynamics of the ``direction" $Y$ of the trajectory  can be no longer described  by properties of suitable matrices.
\end{Remark}

\subsection{The closed-loop equation}\label{ssec:cle} As we will see in the next sections, both  the optimal control for the cooperating agents and the equilibrium strategy for the competing agents in the game described in the forthcoming Section \ref{sec:game}  share the same structure 
\begin{equation}\label{eq:cthetagen}
   c^*_i(t)=\psi_i(X(t))=\theta \langle X(t),\mathbf e\rangle=\theta EX(t)
\end{equation} where $E$ is the $n\times n$ matrix having  the $i$-th row  equal to $\mathbf e$ when $i\in F, $ and equal to the zero vector when $i\not \in F, $ namely
\begin{equation}\label{matrixE}
    E=(e_{ij}),\quad \textrm{with}\ \ e_{ij}=1, \forall i\in F,\ \  \textrm{and}\ e_{ij}=0, \forall i\not\in F, \ \ \forall j\in N. 
\end{equation}
In particular, for the case of cooperating agents, one chooses $\theta=\theta^*$ to have the optimal feedback control. What we say next will hold for all $\theta$ unless otherwise specified.

When we insert a closed-loop control of type  \eqref{eq:cthetagen} into the state equation \eqref{vectsys}, we derive a closed-loop equation of type
\begin{equation}\label{eq:cleoc}\begin{cases}     
      X'(t)=\varphi(\langle \mathbf e, X(t)\rangle )X(t)+(D+B^\top-\theta E)X(t) &t\ge0\\ X(0)=x.
  \end{cases}
  \end{equation} 
  and of which we need to discuss existence and uniqueness. To do so, we note that \eqref{eq:cleoc} is equivalent 
 to the coupled systems for the evolution of the total mass and the stock shares, namely 
\begin{equation}\label{eq:sism^*}
    \begin{cases}        m'(t)=\big(\varphi(m(t))-f\theta\big)m(t), &t\ge0\\
        Y'(t)=(D+B^\top-\theta E+f\theta I)Y(t), &t\ge0\\
         (m(0), Y(0))=(\langle\mathbf e, x\rangle, \frac{x}{ \langle\mathbf e, x\rangle}). 
    \end{cases}
\end{equation}
Clearly, the (linear) equation for $Y(t)$ has a unique solution, and the same holds for the equation for $m(t)$. This implies the following fact.

\begin{Corollary}
    In the assumptions of Theorem \ref{thm:OC}, and when  $\theta=\theta^*$, the closed-loop equation \eqref{eq:cleoc} has a unique solution $X^*(t)$.
\end{Corollary}

\subsection{Total mass dynamics}

\begin{Proposition}[Evolution of the Total Mass] \label{lem:m^*}Assume  Hypothesis $(N1)$, $x\ge0$ with $x\neq0$.   A sol\-ution $m(t)$ of \eqref{eq:sism^*} is:
    \begin{itemize}
    \item[$(i)$] in assumptions $(S1)(U1)$, for $  \Gamma>f\theta$, 
     $$m(t)=\left[e^{-(\Gamma-f\theta)(\sigma-1)t}\big(\langle\mathbf e, x\rangle^{1-\sigma}-\frac 1 K\frac\Gamma{\Gamma-f\theta}\big)+\frac 1 K\frac\Gamma{\Gamma-f\theta}\right]^{\frac 1{1-\sigma}}$$
    
    \item[$(ii)$] in assumptions $(S2)(U1)$, for $f(1-\sigma)<1$, 
    $$m(t)=\left[e^{-(\delta+f\theta)(1-\sigma)t}\big(\langle\mathbf e, x\rangle^{1-\sigma}-\frac 1 {\delta+f\theta}\big)+\frac 1 {\delta+f\theta}\right]^{\frac 1{1-\sigma}}$$
    
    \item[$(iii)$] in assumptions  $(S3)(U2)$, 
          $$m(t)=\exp\left[ e^{-\frac\Gamma K t}\left(\log\langle  \mathbf e, x\rangle-K+\frac{f\theta}\Gamma\right) +K-\frac{f\theta}\Gamma\right] $$
\end{itemize}  Moreover,   $m(t)>0$ at all times $t\ge0$ and converges to a finite positive number $\hat m$. 
    \end{Proposition} 
The proof of this fact is a straightforward adaptation of the proof of Lemma \ref{lem:Y}. 

Taking limits in the explicit formulas of Proposition \ref{lem:m^*} and plugging the planner’s coefficient $\theta^*$ from Table \ref{TabOC} yields the closed-form steady states values for $m$:
\[
\begin{array}{l}
\text{(S1)(U1):}\quad 
m^{*}=\Big[\tfrac{K}{\sigma}\Big(1-\tfrac{\rho}{\Gamma}\Big)\Big]^{\frac{1}{\sigma-1}},\\
\text{(S2)(U1):}\quad 
m^{*}=\Big[\tfrac{\rho+\delta}{\sigma}\Big]^{\frac{1}{\sigma-1}},\\
\text{(S3)(U2):}\quad 
m^{*}=\exp\!\Big(K-\tfrac{\rho}{\Gamma}-\tfrac{1}{K}\Big).
\end{array}
\]
In particular, in the planner benchmark, both $m^{*}$ and the optimal aggregate extraction $f\theta^*$ (being $\theta^*$ proportional to $1/f$, see Table \ref{TabOC}) are independent of the network and of the number of active sites $f$. 

\begin{Remark}
From the previous formulas we have that:
\[
\frac{\partial m^*}{\partial \rho}<0,\qquad
\frac{\partial m^*}{\partial \Gamma}>0\ \text{(where it appears)},\qquad
\frac{\partial m^*}{\partial K}>0\ \text{(where it appears)}
\]
Indeed, as one can aspect, a higher discount rate $\rho$ lowers the present value of future shifting extraction upward and the steady state downward while a larger intrinsic growth $\Gamma$ or carrying capacity $K$ raises the regeneration curve, allowing a higher sustainable stock.
\end{Remark}

\subsection{Stock Shares Dynamic}
Clearly,  the components of the limit value of the trajectory, when existing, may bear different signs although  their sum $m$ is positive. We then proceed to study the stock shares vector $Y(t)$. To this extent, the following lemma is fundamental.

\begin{Lemma} \label{lem:eigenBB} Let $E$ be the matrix defined in \eqref{matrixE}, and  $\theta\in\mathbb R$. Then: 
		\begin{itemize}
			\item[$(i)$] the matrix $(D+B - \theta  E^\top+f \theta I)$ has    eigenvector $\mathbf e$  associated with  the null eigenvalue; hence, there exists a real  eigenvector $\zeta_\theta$ of $(D+B - \theta  E^\top+f\theta I)^\top$ associated with  the null eigenvalue. 
			\item[$(ii)$] Consider a basis $\{ \zeta, v_2,\dots,v_n\}$ of generalized eigenvectors  of $D+B^T$ , associated with the eigenvalues $\{0,\lambda_2,\dots,\lambda_n\}$ described in Lemma \ref{lem:D+B}. Then $\{ \zeta_\theta, v_2,\dots,v_n\}$ is a basis of generalized eigenvectors for $(D+B -\theta  E^\top+f\theta I)$  associated with eigenvalues $\{0,\lambda_2+f\theta,\dots,\lambda_n+f \theta\}$. 
		\end{itemize}		
	\end{Lemma}
The proof of this fact can be found in  Appendix \ref{subapp:otherproofs}.

Note that, in comparison to the system with null extraction, the equation for $Y(t)$ is modified by a linear term $(-\theta E+f\theta )Y(t)$, with the difference that now the matrix $D+B^\top-\theta E+f\theta$ is not necessarily Metzler, and the associated properties may fail to be satisfied, meaning the optimal trajectory may fail to lie entirely in the positive orthant or may cease to converge towards the direction of the (possibly) dominant eigenvector $\zeta_\theta$.


\bigskip

Lemma \ref{lem:eigenBB}  implies that, with respect to the shares $Y(t)$ described  in Lemma \ref{lem:Y}, the extraction process modifies the direction of only the dominant eigenvector $\zeta$ of the matrix $D+B^\top$, which changes from $\zeta$ to $\zeta^*\equiv\zeta_{\theta^*}$, and the only eigenvalues $\lambda_i$ for $i\ge2$, which are increased by $f\theta^*$. Notably, the remaining eigenvectors remain the same.

Note also that $\textrm{Re}\lambda_i+f\theta^*$ may be become positive, so that 0 ceases to be the eigenvalue of maximal real part; in addition,  the matrix $D+B -\theta  E^\top+f\theta^* I$ may cease to be Metzler, so that without further assumptions we cannot establish whether $\zeta^*$ remains a positive (dominant) eigenvector. It emerges that both of these issues do not ensue when players are sufficiently ``patient", meaning $\theta$ is small enough. 
To this extent, we set 
$$\theta_1\equiv \frac{\vert\lambda_2\vert}{f} $$
and note that $\theta<\theta_1$ if and only if $\lambda_2+\gre{f}\theta$ is smaller than the dominant eigenvalue 0. Moreover, since   $\zeta_\theta$ depends continuously on $\theta$, and $\zeta>0$, then  there exists $\theta_2>0$ such that $ \zeta_\theta$ is   positive for all $\theta<\theta_2$. We then set
 \begin{equation*}
 \theta_2=\sup\{r:\zeta_\theta>0, \forall   \theta\in[0,r]\}.
\end{equation*} 


This discussion therefore leads us to the first result concerning the admissibility of the candidate-optimal control identified in feedback form in Theorem \ref{thm:OC}:

\begin{Proposition}
\label{pr:lrs} 
Assume to be in the hypotheses of Theorem \ref{thm:OC} and suppose that $0<\theta<\min\{\theta_1, \theta_2\}$. Let $X, Y$ be the state variables described in \eqref{eq:sism^*}. Then:
\begin{itemize}
    \item[$(i)$] 
 $\zeta_\theta>0$ and 
\[
\lim_{t\to\infty}Y(t)=\frac{\zeta_\theta}{\langle e,\zeta_\theta\rangle},\qquad
\lim_{t\to\infty}X(t)=\hat m\,\frac{\zeta_\theta}{\langle e,\zeta_\theta\rangle},
\]
where $\hat m=\lim_{t\to\infty} m(t)$ is given in Proposition \ref{lem:m^*}.

\item[$(ii)$] There exists a constant $L$ such that, for every $x$ in the cone
\[
\mathcal{C} := \left \{ x\in \mathbb{R}^n_+ \; : \; \left \vert \frac{x}{\left \langle x, \mathbf e\right\rangle} - \frac{\zeta_\theta}{\left \langle \zeta_\theta, \mathbf e\right\rangle}  \right \vert < L \right \},
\]
the candidate-optimal trajectory is admissible and therefore optimal.
\end{itemize}
\end{Proposition}
The proof of this fact can be found in  Appendix \ref{subapp:otherproofs}.

The result of Proposition \ref{pr:lrs} ensures that, for a certain subset of initial data, the candidate-optimal trajectory is admissible. One may, on the other hand, ask whether there exist conditions under which the candidate-optimal trajectory is admissible starting from any initial data in $\mathbb{R}_n^+$. This question will be related to the conditions used in Section \ref{sec:game} to ensure that the game is well defined and sub-game perfect.

The idea is that the points where one actually needs to check that the system remains positive are those on the boundary of $\mathbb{R}_n^+$. If, at such points, the direction of the trajectory predicted by the candidate-optimal control is inward, then there is no possibility of leaving the positive orthant. We only need to check the nodes with extraction, since for the others, when the stock is zero, there is no outflow of the resource. Hence, for every $i\in F$, it is required that $Y'_i$, as predicted by the second equation in (\ref{eq:sism^*}), be positive starting from any $Y$ such that $Y_i=0$ (and the other components non-negative, with at least one positive). Since the second equation in (\ref{eq:sism^*}) is linear, this condition is equivalent to requiring that 
\[
\left\langle (D+B^\top-\theta E+f\theta I)e_j, e_i \right\rangle >0 
\]
for every $j\neq i$. In other words, every element of the matrix $(D+B^\top-\theta E+f\theta I)$ must be positive for all $i\in F$ and $j\neq i$, that is, 
\[
(0<)\theta^* < \min_{i \in F, j\neq i} b_{ij} 
\]
(so this implies, in particular, that $\min_{i \in F, j\neq i} b_{ij} >0$).

We can summarize this discussion in the following proposition.

\begin{Proposition}
\label{pr:admissibility2}
Suppose that the conditions of Theorem \ref{thm:OC} are satisfied and that 
\[
\theta^* < \min_{i \in F, j\neq i} b_{ij}. 
\]
Then the candidate-optimal trajectory is admissible for any initial data in $\mathbb{R}^n_+$ and is optimal.
\end{Proposition}

\section{Strategic interaction}
\label{sec:game}

In the investigation of the Nash equilibria of the game, we restrict our attention to stationary Markovian equilibria. Specifically, the extraction rates $c_i$ of the agents are modeled as reaction maps to the observed stock level $X(t)$ at time $t$:
$$c(t) = \psi(X(t)),\ \ \text{with}\ c_i(t) = \psi_i(X(t)),\ \forall i \in N$$
(where $\psi_i \equiv 0$ for all $i \in N \setminus F$). Consequently, the system evolves according to the \emph{closed-loop equation} (CLE):
\begin{equation}\label{eq:cle}
    \begin{cases}
        \dot{X}(t) = \varphi \left( \langle X(t),\mathbf e\rangle\right)X(t)+(D + B^\top) X(t) - \psi(X(t)), & t > 0 \\ 
        X(0) = x, & 
    \end{cases}
\end{equation}
provided that such an equation admits a unique solution.
More precisely,
$$\psi = (\psi_1, \psi_2, \dots, \psi_n), \ \ \psi_i: \mathbb{R}_+^n \to [0, +\infty).$$
We denote the set of admissible strategy profiles by
$$\mathbb{A} = \mathbb{A}_1 \times \mathbb{A}_2 \times \cdots \times \mathbb{A}_n,$$
where $\mathbb{A}_i$ represents the collection of all (re)actions $\psi_i$ of Player $i$ (or a null reaction at nodes with reserves).
We denote the solution of equation (\ref{eq:cle}) by $X(t;x, \psi)$. We also follow the convention of denoting by $\psi_{-i}$ all components of $\psi$ other than the $i$-th, so that $\psi = (\psi_i, \psi_{-i})$.

\begin{Definition}\label{MNE}\emph{(Markovian Perfect Equilibrium)} 
We define a strategy profile $\psi \in \mathbb{A}$ as a \emph{Markovian Perfect Equilibrium (MPE)} if, for all $x_0 \in \mathbb{R}_+^n$ and $i \in F$, the control $c_i(t) = \psi_i(X_i^{\psi, x_0}(t))$ is optimal for Player $i$’s problem, 
where $c_j(t) = \psi_j(X(t))$ for every $j \neq i$, the nonnegative state and extraction rate constraints \eqref{constr}, and the discounted total payoff $J_i(c_i)$ given by \eqref{utility}, which is to be maximized over the set of admissible controls $\mathbb{A}_i$.
\end{Definition}

We focus on network with the following connections characteristics:  all nodes in $F$ have positive inward edges from all other nodes: $g_{ij}>0$,  $\forall i \in F$, $j\neq i$.

We consider the following set of strategies  $\mathbb A_i$ for player $i$,   $i\in F$, given by 
\begin{equation}\label{def:Abuffoi}
\mathbb A_i := \left \{ \psi_i \colon \mathbb{R}_+^n \to [0, +\infty) \; : \; \begin{array}{l} 
(i)\, \psi_i \text{ is Lipschitz-continuous }\\
(ii)\, \psi_i(x) \leq \left \langle ( D+B^\top)x, e_i \right\rangle\\ 
\text{ \qquad for all $x \in \mathbb{R}_+^n$ such that $x_i=0$.}
\end{array} 
\right \}
\end{equation}
When $i\not\in F$, we assume $\mathbb A_i$ contains only the null strategy. Note that the Lipschitz-continuity of $\psi_i$ implies that the CLE (\ref{eq:cle}) has a unique solution, $X(t)$, whereas the condition $(ii)$ ensures that, when the stock at node $i$ is null, the extraction $\psi_i$ can be, at most, as much as the inflow at  $i$ from the other nodes, so that the stock remains nonnegative at all times. For this reason a Markovian Nash equilibrium $\psi$ within $\mathbb{A}$ is inherently subgame perfect because if a player deviates (whether intentionally or by mistake) from $\psi$, they cannot exit the set $\mathbb{R}_+^n$, and the strategy profile $\psi$ remains valid as a Nash equilibrium from any state reached within $\mathbb{R}_+^n$.

\subsection{Existence of Symmetric Markovian equilibria}
In Theorem \ref{th:ME}, we provide an existence result for a Markovian equilibrium (i.e. a Nash equilibrium in closed-loop form) for the game where all players extract the same resource quantity and gain an equal utility.
 Explicit formulas will also be provided. The general statement of the theorem will refer to any of the frameworks described in Table \ref{Tabellone} below. 

 \begin{small} 
\begin{table} [h]
\begin{center}
\caption{ }

\label{Tabellone}
\begin{tabular}{c|c|c|c|c|c  }
framework  &
$\varphi(m)$&
$u(c)$ &
$\hat \theta$&
$A$&
$B$  
 \\
\hline\hline
 (S1)(U1)&
$\Gamma \left( 1-\frac{1}{K}m^{\sigma-1}\right) $&
$-\frac{c^{1-\sigma}}{\sigma-1}$, $\sigma>1$&
$\frac{\rho+\Gamma(\sigma-1)}{1+f(\sigma-1)}$&
${\hat \theta^{-\sigma}} $&
$-\frac{\Gamma \hat \theta^{-\sigma} }{K\rho}$ \\
\hline
  (S2)(U1) & 
$m^{\sigma -1}-\delta$&
$\frac{c^{1-\sigma}}{1-\sigma}$, $\sigma<1$ &
$\frac{\rho+\delta(1-\sigma)}{1-f(1-\sigma)} $&
$ {\hat \theta^{-\sigma}}$&
$\frac {\hat \theta^{-\sigma}}{\rho}$ \\
\hline
 (S3)(U2)&
$\Gamma \left( 1-\frac{1}{K}\log m\right) $&
$\log c$ &
$\rho + \frac \Gamma K $ &
$\hat \theta^{-1}$&
$ \frac {\Gamma K-f\hat \theta+\hat \theta\log\hat \theta}{\hat \theta\rho }  $ 
 \end{tabular}
\end{center}
\end{table}
\end{small}

\begin{Theorem}[Existence of Markovian Equilibria]\label{th:ME}
In each of the assumption frameworks described in Table \ref{Tabellone}, and assuming\footnote{Note that in order for $\hat \theta$ to be strictly positive, in framework  (S2)(U1) it must be assumed that  $f<\frac1{1-\sigma}$. Conversely,   in  frameworks (S1)(U1) and (S3)(U2) $\hat \theta$ is positive without further assumptions.} 
\[
0 < \hat \theta < \min_{i\in F, \, i\neq j} b_{ij},
\]
(being $b_{ij}$ the components of $B$) the following facts hold:
\begin{itemize}
\item[$(i)$] the set of strategies $\psi^*$   given by
\begin{equation}\label{eq:optccl}
 \psi_{i}^*=0, \ \  \forall i\not\in F \quad\textrm{and}  \quad  \psi_{i }^*(t)=\hat \theta\langle \mathbf e, X(t)\rangle,\ \forall i\in F;
\end{equation}
is admissible (i.e. it belongs to $\mathbb{A}$)
    \item[$(ii)$] the set of strategies $\psi^*$ described above is
a Markovian Perfect Equilibrium
 \item[$(iii)$]   the utility of each player at the equilibrium is the same
\begin{equation}\label{eq:vf}
V^i(x)=
 V(x):=A\;u(\langle\mathbf e,x\rangle)+B
\end{equation}

\end{itemize}
    \end{Theorem}
    The proof of the theorem can be found in Appendix \ref{subapp:otherproofs}.

\bigskip

We can compare extraction in the planner’s case and in the decentralized game, and obtain a classic “tragedy of the commons” result. Indeed, if we denote by $\theta^*$ the per–site feedback under the planner (Table \ref{TabOC}) and by $\hat\theta$ the per-site coefficient in the symmetric Markov equilibrium (Table \ref{Tabellone}), the difference in aggregate extraction is
\[
\Delta_f\;:=\;f\hat\theta-f\theta^* .
\]
Under the interiority requirements of each framework, one obtains the closed forms:
\[
\begin{array}{ll}
\text{(S1)(U1)}  &\quad 
\Delta_f=\dfrac{(\rho+\Gamma(\sigma-1))(f-1)}{\sigma\,[1+f(\sigma-1)]},\\[2.2ex]
\text{(S2)(U1)}  
&\quad 
\Delta_f=\dfrac{(\rho+\delta(1-\sigma))(f-1)}{\sigma\,[1-f(1-\sigma)]},\\[2.2ex]
\text{(S3)(U2)}  
&\quad 
\Delta_f=(f-1)\!\left(\rho+\frac{\Gamma}{K}\right),
\end{array}
\]
which is positive in all three cases.

Similarly, we denoted by $m^*$ the long-run stock of the resource in the planner’s case, and we can denote by $\hat m$ the long-run stock in the decentralized game. We have in the steady state $\varphi(m)=f\hat\theta$ so, using the corresponding expressions for $\varphi$ and $\hat\theta$, we obtain:
\[
\begin{aligned}
\text{(S1)(U1)}\ &\ 
\hat m=\Big[\tfrac{K}{1+f(\sigma-1)}\Big(1-\tfrac{f\rho}{\Gamma}\Big)\Big]^{\frac{1}{\sigma-1}};\\[1ex]
\text{(S2)(U1)}\ &\ 
\hat m=\Big[\tfrac{\delta+f\rho}{\,1-f(1-\sigma)\,}\Big]^{\frac{1}{\sigma-1}};\\[1ex]
\text{(S3)(U2)}\ &\ 
\hat m=\exp\!\Big(K-\tfrac{f\rho}{\Gamma}-\tfrac{f}{K}\Big).
\end{aligned}
\]
It is then easy to verify that, in all three cases, $\hat m<m^*$ for $f>1$ (and they coincide for $f=1$).

\section{Conclusions}
\label{sec:conclusions}

We have provided an explicit, network-based theory of renewable resource extraction with mass-conserving migration and nonlinear (concave) growth. By pairing three canonical growth laws (logistic, power, log-type) with related CRRA/log utilities, we derived closed-form value functions and feedback rules for a planner and constructed a symmetric Markov equilibrium on strongly connected graphs. Perron–Frobenius geometry continues to govern long-run spatial allocations through value transport, while nonlinearity disciplines aggregate dynamics and ties global mass to its spatial distribution. These formulas yield transparent comparative statics, on network centrality (dominant eigenvector and spectral gap), growth curvature, discounting, and the set of active nodes, and quantify the efficiency wedge between decentralized and coordinated extraction. Linear benchmarks emerge as limiting cases.

{\small 
			\bibliographystyle{apalike}
			\bibliography{biblio2}

\begin{thebibliography}{}

\bibitem[Albeverio and Mastrogiacomo, 2022]{albeverio2022large}
Albeverio, S. and Mastrogiacomo, E. (2022).
\newblock Large deviation principle for spatial economic growth model on networks.
\newblock {\em Journal of Mathematical Economics}, 103:102784.

\bibitem[Bhat and Huffaker, 2007]{BhatHuffaker07}
Bhat, M.~G. and Huffaker, R.~G. (2007).
\newblock Management of a transboundary wildlife population: A self-enforcing cooperative agreement with renegotiation and variable transfer payments.
\newblock {\em Journal of Environmental Economics and Management}, 53(1):54--67.

\bibitem[Boucekkine et~al., 2013]{BCF13}
Boucekkine, R., Camacho, C., and Fabbri, G. (2013).
\newblock Spatial dynamics and convergence: The spatial ak model.
\newblock {\em Journal of Economic Theory}, 148(6):2719--2736.

\bibitem[Boucekkine et~al., 2019]{boucekkine2019growth}
Boucekkine, R., Fabbri, G., Federico, S., and Gozzi, F. (2019).
\newblock Growth and agglomeration in the heterogeneous space: a generalized ak approach.
\newblock {\em Journal of Economic Geography}, 19(6):1287--1318.

\bibitem[Boucekkine et~al., 2022a]{boucekkine2022dynamic}
Boucekkine, R., Fabbri, G., Federico, S., and Gozzi, F. (2022a).
\newblock A dynamic theory of spatial externalities.
\newblock {\em Games and Economic Behavior}, 132:133--165.

\bibitem[Boucekkine et~al., 2022b]{boucekkine2022managing}
Boucekkine, R., Fabbri, G., Federico, S., and Gozzi, F. (2022b).
\newblock Managing spatial linkages and geographic heterogeneity in dynamic models with transboundary pollution.
\newblock {\em Journal of Mathematical Economics}, 98:102577.

\bibitem[Brito, 2004]{Brito04}
Brito, P. (2004).
\newblock {The Dynamics of Growth and Distribution in a Spatially Heterogeneous World}.
\newblock Technical report.
\newblock ISEG Working Papers 2004/14.

\bibitem[Brito, 2022]{brito2022dynamics}
Brito, P.~B. (2022).
\newblock The dynamics of growth and distribution in a spatially heterogeneous world.
\newblock {\em Portuguese Economic Journal}, 21(3):311--350.

\bibitem[Brock et~al., 2014]{BXEY14}
Brock, W., Xepapadeas, A., and Yannacopoulos, A.~N. (2014).
\newblock Optimal control in space and time and the management of environmental resources.
\newblock {\em Annual Review of Resource Economics}, 6:33--68.

\bibitem[Calvia et~al., 2024]{calvia2024optimal}
Calvia, A., Gozzi, F., Leocata, M., Papayiannis, G.~I., Xepapadeas, A., and Yannacopoulos, A.~N. (2024).
\newblock An optimal control problem with state constraints in a spatio-temporal economic growth model on networks.
\newblock {\em Journal of Mathematical Economics}, page 102991.

\bibitem[Clark, 2010]{Clark2010}
Clark, C.~W. (2010).
\newblock {\em Mathematical Bioeconomics: The Mathematics of Conservation}.
\newblock Wiley-Blackwell, Hoboken, NJ, 3rd edition.

\bibitem[Costello and Kaffine, 2018]{CostelloKaffine18}
Costello, C. and Kaffine, D. (2018).
\newblock Natural resource federalism: Preferences versus connectivity for patchy resources.
\newblock {\em Environmental and Resource Economics}, 71(1):99--126.

\bibitem[Costello and Polaski, 2008]{CostPol08}
Costello, C. and Polaski, S. (2008).
\newblock Optimal harvesting of stochastic spatial resources.
\newblock {\em Journal of Environmental Economics and Management}, 56:1--18.

\bibitem[Costello et~al., 2015]{CTQ15}
Costello, C., Qu{\'e}rou, N., and Tomini, A. (2015).
\newblock Partial enclosure of the commons.
\newblock {\em Journal of Public Economics}, 121:69--78.

\bibitem[Courchamp et~al., 2008]{CourchampBerecGascoigne2008}
Courchamp, F., Berec, L., and Gascoigne, J. (2008).
\newblock {\em Allee Effects in Ecology and Conservation}.
\newblock Oxford University Press, Oxford.

\bibitem[de~Frutos et~al., 2022]{de2022investment}
de~Frutos, J., Gat{\'o}n, V., L{\'o}pez-P{\'e}rez, P.~M., and Mart{\'\i}n-Herr{\'a}n, G. (2022).
\newblock Investment in cleaner technologies in a transboundary pollution dynamic game: A numerical investigation.
\newblock {\em Dynamic Games and Applications}, 12(3):813--843.

\bibitem[De~Frutos and Martin-Herran, 2019]{de2019spatial}
De~Frutos, J. and Martin-Herran, G. (2019).
\newblock Spatial effects and strategic behavior in a multiregional transboundary pollution dynamic game.
\newblock {\em Journal of Environmental Economics and Management}, 97:182--207.

\bibitem[Fabbri, 2016]{fabbri2016geographical}
Fabbri, G. (2016).
\newblock Geographical structure and convergence: A note on geometry in spatial growth models.
\newblock {\em Journal of Economic Theory}, 162:114--136.

\bibitem[Fabbri et~al., 2020]{FFF20}
Fabbri, G., Faggian, S., and Freni, G. (2020).
\newblock Policy effectiveness in spatial resource wars: A two-region model.
\newblock {\em Journal of Economic Dynamics and Control}, 111:103818.

\bibitem[Fabbri et~al., 2024]{FFF24}
Fabbri, G., Faggian, S., and Freni, G. (2024).
\newblock On competition for spatially distributed resources in networks.
\newblock {\em Theoretical Economics}, 19(2):743--781.

\bibitem[Fabbri et~al., 2025]{FFF25}
Fabbri, G., Faggian, S., and Freni, G. (2025).
\newblock Growth models with externalities on networks.
\newblock {\em Decisions in Economics and Finance}, 48(1):415--436.

\bibitem[Farina and Rinaldi, 2000]{FR2000}
Farina, L. and Rinaldi, S. (2000).
\newblock {\em Positive Linear Systems: Theory and Applications}.
\newblock Wiley-Interscience, New York.

\bibitem[Gaudet and Lohoues, 2007]{GL2007}
Gaudet, G. and Lohoues, H. (2007).
\newblock On limits to the use of linear markov strategies in common property natural resource games.
\newblock {\em Environmental Modeling and Assessment}, 13(4):567--574.

\bibitem[Gordon, 1954]{Gordon1954}
Gordon, H.~S. (1954).
\newblock The economic theory of a common-property resource: The fishery.
\newblock {\em Journal of Political Economy}, 62(2):124--142.

\bibitem[Hanski, 1999]{Hanski1999}
Hanski, I. (1999).
\newblock {\em Metapopulation Ecology}.
\newblock Oxford University Press, Oxford.

\bibitem[Herrera et~al., 2016]{HerreraEtAl2016}
Herrera, G., Moeller, H.~V., Neubert, M.~G., et~al. (2016).
\newblock High-seas fish wars generate marine reserves.
\newblock {\em Theoretical Ecology}, 9(2):229--247.

\bibitem[Hilborn and Walters, 1992]{HilbornWalters1992}
Hilborn, R. and Walters, C.~J. (1992).
\newblock {\em Quantitative Fisheries Stock Assessment: Choice, Dynamics and Uncertainty}.
\newblock Chapman \& Hall, London.

\bibitem[Horn and Johnson, 2013]{HornJohnson2013}
Horn, R.~A. and Johnson, C.~R. (2013).
\newblock {\em Matrix Analysis}.
\newblock Cambridge University Press, Cambridge, 2nd edition.

\bibitem[{\.I}lk{\i}l{\i}\c{c}, 2011]{Ilkilic2011}
{\.I}lk{\i}l{\i}\c{c}, R. (2011).
\newblock Network formation and strategic resource extraction.
\newblock {\em Journal of Public Economic Theory}, 13(6):885--915.

\bibitem[Janmaat, 2005]{janmaat2005sharing}
Janmaat, J.~A. (2005).
\newblock Sharing clams: tragedy of an incomplete commons.
\newblock {\em Journal of Environmental Economics and Management}, 49(1):26--51.

\bibitem[Kaffine and Costello, 2011]{KaffineCostello11}
Kaffine, D.~T. and Costello, C. (2011).
\newblock Unitization of spatially connected renewable resources.
\newblock {\em The BE Journal of Economic Analysis and Policy}, 11(1).

\bibitem[Murray, 2003]{Murray2003}
Murray, J.~D. (2003).
\newblock {\em Mathematical Biology {II}: Spatial Models and Biomedical Applications}.
\newblock Springer, New York, 3rd edition.

\bibitem[Officer et~al., 1982]{Officer1982}
Officer, C.~B., Smayda, T.~J., and Mann, R. (1982).
\newblock Benthic filter feeding: A natural eutrophication control.
\newblock {\em Marine Ecology Progress Series}, 9:203--210.

\bibitem[Okubo and Levin, 2001]{OkuboLevin2001}
Okubo, A. and Levin, S.~A. (2001).
\newblock {\em Diffusion and Ecological Problems: Modern Perspectives}, volume~14 of {\em Interdisciplinary Applied Mathematics}.
\newblock Springer, New York, 2nd edition.

\bibitem[Plourde and Yeung, 1989]{PY1989}
Plourde, C. and Yeung, D. (1989).
\newblock Harvesting of a transboundary replenishable fish stock: A noncooperative game solution.
\newblock {\em Marine Resource Economics}, 6(1):57--70.

\bibitem[Ricci, 2024]{ricci2024non}
Ricci, C. (2024).
\newblock A non-invariance result for the spatial ak model.
\newblock {\em Decisions in Economics and Finance}, pages 1--20.

\bibitem[Sanchirico and Wilen, 2005]{SCW05}
Sanchirico, J.~N. and Wilen, J.~E. (2005).
\newblock Optimal spatial management of renewable resources: matching policy scope to ecosystem scale.
\newblock {\em Journal of Environmental Economics and Management}, 50:23--46.

\bibitem[Santambrogio et~al., 2020]{SXY17}
Santambrogio, F., Xepapadeas, A., and Yannacopoulos, A. (2020).
\newblock Rational expectations equilibria in a ramsey model of optimal growth with non-local spatial externalities.
\newblock {\em Journal de Math{\'e}matiques Pures et Appliqu{\'e}es}, 140:259--279.

\bibitem[Schaefer, 1954]{Schaefer1954}
Schaefer, M.~B. (1954).
\newblock Some aspects of the dynamics of populations important to the management of commercial marine fisheries.
\newblock {\em Inter-American Tropical Tuna Commission Bulletin}, 1(2):27--56.

\bibitem[Schaefer, 1957]{Schaefer1957}
Schaefer, M.~B. (1957).
\newblock Some considerations of population dynamics and economics in relation to the management of commercial marine fisheries.
\newblock {\em Journal of the Fisheries Research Board of Canada}, 14(5):669--681.

\bibitem[Smith et~al., 2009]{SSW09}
Smith, M.~D., Sanchirico, J.~N., and Wilen, J.~E. (2009).
\newblock The economics of spatial-dynamic processes: Applications to renewable resources.
\newblock {\em Journal of Environmental Economics and Management}, 57:104--121.

\bibitem[Viana et~al., 2019]{CostelloEtAl2019}
Viana, D.~F., Costello, C., and White, C. (2019).
\newblock Management of mobile species with spatial property rights.
\newblock {\em Fisheries Research}, 213:203--213.
\newblock Key moved to match manuscript citation.

\bibitem[Westley et~al., 2013]{WestleyQuinnDittman2013}
Westley, P. A.~H., Quinn, T.~P., and Dittman, A.~H. (2013).
\newblock Rates of straying by hatchery-produced pacific salmon (oncorhynchus spp.) and steelhead (o. mykiss) differ among species, life history types, and populations.
\newblock {\em Canadian Journal of Fisheries and Aquatic Sciences}, 70(5):735--746.

\bibitem[Xepapadeas, 2010]{XE10}
Xepapadeas, A. (2010).
\newblock The spatial dimension in environmental and resource economics.
\newblock {\em Environment and Development Economics}, 15(6):747--758.

\end{thebibliography}
		}

\newpage

\appendix

\begin{small}
    
\section{Technical Proofs}
		\label{app:proofs}
        \subsection{Proofs for Section \ref{sec:model}}
 \begin{proof}[\textbf{Proof of Lemma \ref{lem:D+B}}]
The matrix \(D+B \) has (at least) one zero eigenvalue. In fact, the sum of the elements along the rows is 0, which implies that the columns are linearly dependent: their sum is indeed the zero vector. Consequently, the vector $\mathbf e$ is a (positive)  eigenvector associated to the eigenvalue 0. Note also that the Gershgorin Circle Theorem (Theorem 6.1.1 in \citealp{HornJohnson2013}) implies that all other eigenvalues $\lambda_i$ satisfy $\textrm{Re}\lambda_i\le0$. Now since the network is strongly connected and the matrix $D+B$ is Metzler, the Perron-Frobenius theorem implies $D+B$ has a dominant real eigenvalue \( \lambda \) associated with a dominant eigenvector $w$, such that $\lambda> \textrm{Re}\lambda_i$ for all other eigenvalues $\lambda_i$, and with $w$  the only eigenvector whose coordinates can be chosen all strictly positive. Then necessarily   \eqref{eq:eigen} holds with $\lambda=0$,  and $w$ is proportional to $\mathbf e$.  Since the transposed matrix has the same eigenvalues then there exists a dominant eigenvector $\zeta$ for \(D+B^T\) for the dominant eigenvalue $0$.

\end{proof}

\begin{proof}[\textbf{Proof of Lemma \ref{lem:Y}}]
The proof of the explicit formula for $Y(t)$ in \eqref{eq:Yformula} is a direct application of variation of constants formula. The positivity of the $Y_i$'s is a consequence of the matrix $D+B^\top$ being Metzler (see for example \cite{FR2000}).

By Lemma \ref{lem:D+B}, $D+B^\top$ and $(D+B^\top)^\top$ share the spectrum $\{0,\lambda_2,\dots,\lambda_n\}$ with $\mathrm{Re}\,\lambda_i<0$ for $i\ge 2$. Let $\zeta\gg 0$ be the right eigenvector of $D+B^\top$ for eigenvalue $0$, and note that $\mathbf e$ is a left eigenvector of $D+B^\top$ for eigenvalue $0$. The spectral projection onto $\ker (D+B^\top)$ is
\[
P=\frac{\zeta\,\mathbf e^\top}{\langle \mathbf e,\zeta\rangle}.
\]
Since $0$ is simple and all other eigenvalues have strictly negative real part, the Jordan decomposition (see, e.g., \citealp{HornJohnson2013}) yields
\[
e^{(D+B^\top)t}=P+R(t),\qquad \|R(t)\|\le C e^{-\eta t}
\]
for some $C,\eta>0$. Therefore,
\[
e^{(D+B^\top)t}x \;\longrightarrow\; P x
=\frac{\langle \mathbf e,x\rangle}{\langle \mathbf e,\zeta\rangle}\,\zeta,
\quad\text{and thus}\quad
Y(t)=\langle \mathbf e,x\rangle^{-1}e^{(D+B^\top)t}x
\;\longrightarrow\; \langle \mathbf e,\zeta\rangle^{-1}\zeta,
\]
which is \eqref{eq:Yformula2}.

    To prove the formulas for $m(t)$ we operate in the equation the change of variables: for $(i)$ and  $(ii)$, $\mu(t)=m(t)^{1-\sigma}$;  for $(iii)$,  $\mu(t)=\log(m(t))$. 
        Hence note that $\mu(t)$ satisfies a linear equation that can be solved by variation of constants formula, thus   the thesis follows.
\end{proof}


  \subsection{Proof of Theorem \ref{thm:OC}. Verification and Uniqueness for the HJB equation} \label{ssec:verif}

For the proof of Theorem \ref{thm:OC}, we make use of several preliminary results, some of which bear independent mathematical interest, and that we recollect in this section.

We start by making the HJB equation \eqref{hjbgen} explicit in the different assumption frameworks:

\begin{itemize}
    \item[(a)] if
    $\sigma\neq1$, and (U1)  (S1) hold, then  
    $$\rho v(x)= \frac{\sigma}{1-\sigma}\sum_{i\in F} \left( \frac{\partial v}{\partial x_i}(x)\right)^{1-\frac1\sigma}  + \langle   \nabla v(x), \alpha(x)\rangle+ \Gamma\left(1-K ^{-1}
    \langle x,\mathbf e\rangle^{\sigma-1}\right)\langle  \nabla v(x), x\rangle$$
    \item[(b)] if $\sigma\neq1$, and   (U1)  (S2) hold,   
    $$\rho v(x)= \frac{\sigma}{1-\sigma}\sum_{i=1}^f \left( \frac{\partial v}{\partial x_i}(x)\right)^{1-\frac1\sigma}  + \langle   \nabla v(x), \alpha(x)\rangle+  \left( 
    \langle x,\mathbf e\rangle^{\sigma-1}-\delta\right)\langle  \nabla v(x), x\rangle$$
    \item[(c)] if $\sigma=1$, and (U2)(S3) hold,     
    $$\rho v(x)=-  \sum_{i=1}^f \ln\left(\frac{\partial v}{\partial x_i}\right)  -f+ \langle   \nabla v(x), \alpha(x)\rangle +  \Gamma\left(1-K ^{-1}
    \ln(\langle x,\mathbf e\rangle)\right)\langle  \nabla v(x), x\rangle$$
\end{itemize}

The next step is  guessing  a solution to the HJB equation.

\begin{Proposition}\label{prop:w} In each of the assumption frameworks in Table \ref{TabOC}
\begin{equation}\label{eq:solhjb}
 w(x)=A\;u(\langle\mathbf e,x\rangle)+B.
\end{equation}
is a classical solution in $\mathbb R^n_+-\{0\}$ of the HJB equation \eqref{hjbgen}.

\end{Proposition}
\begin{proof} 
     Note that  $\nabla  w(x)=Au'(\langle x, \mathbf e\rangle)\mathbf e$. We plug such formula into the associated HJB equation, use \eqref{optcontgen}, and note that  $(D+B)\mathbf e=0$ implies
$$\langle \nabla w(x), \alpha(x)\rangle=Au'(\langle x, \mathbf e\rangle)\langle (D+B)\mathbf e,x\rangle=0.$$
Thus, with easy but tedious calculations that
$w$ is a solution to the HJB equation if and only if  $A$ and $B$ are  those described in Table \ref{TabOC}. 
\end{proof}

For the reader convenience we state here the well known fundamental identity on which verification techniques are based. To this extent we set
$$h(c,p)= \sum_{i=1}^f\left(u\left(c_i\right) - c_i p_i\right) ,\quad  H(p)= \sup _{c > 0}h(c,p).$$

\begin{Lemma}[The Fundamental Identity]\label{thm:fundid} Assume $v(x)$ is any classical solution to the HJB equation \eqref{hjbgen}. Then 
 \begin{align*}
    v(x)=\int_0^t\big[H( \nabla v(X(s)))-&h(c(s), \nabla v(X(s)))\big]ds+\int_0^te^{-\rho t}\sum_{i\in F}u(c_i(s))ds+e^{-\rho t}v(X(t)).
\end{align*}
holding for every $t$,   for every control $c\in\mathbb A$ and trajectory $X(t)$.    
\end{Lemma}
\begin{proof}
One has
\begin{align*}
        \frac{d}{ds}&\left(e^{-\rho s}v(X(s))\right)=-\rho e^{-\rho s}v(X(s))+e^{-\rho t}\langle \nabla v(X(s)),\dot X(s)\rangle\\
        &=e^{-\rho s}\left[-H( \nabla v(X(s)))+ h(c(s),  \nabla v(X(s)))\right]-e^{-\rho s}\sum_{i\in F}u(c_i(s)) 
    \end{align*}
and, integrating on $[0,t]$, obtains the sought identity.
\end{proof}

In the theorems below, we implement verification techniques, proving alongside the uniqueness of the solution of the  HJB equation \eqref{hjbgen}.
We preliminary note that,  there exists a constant $C>0$ such that, for every  initial condition $x\in\mathbb R^n-\{0\}$ and for every admissible control $c\in\mathbb A$,  the associated  trajectory $X(t)=X(t; x,c)$ satisfies
 \begin{equation}\label{eq:growth_X}
\vert X(t)\vert\leq C  (1+\vert x\vert).
 \end{equation}
Indeed the solution $m(t)$ of \eqref{stateqgen} and that for a null extraction, $m_0(t)$, defined in Lemma \ref{lem:Y} satisfy 
\begin{equation}
    \label{eq:bound-for-m}
m(t)\le m_0(t)\le\max\left\{ \sum_ix_i, \bar m\right\}\le C(1+\vert x\vert)
\end{equation}
where the second inequality is implied by the monotonic convergence, from above or below, of  
$m_0(t)$ starting at $m_0(0)=\sum_ix_i$ to the long-run mass $\bar m$ (where in (S1), $\bar m=K^\frac 1{\sigma-1}$; in (S2), $\bar m=e^K$; in (S3), $\bar m=\delta^{-\frac 1{\sigma-1}}$), while the last inequality is true with $C=\sqrt n\max\{1, \bar m\}.$ 
\noindent Hence, since $X_i(t)\ge0$ for all $i$, 
$$\vert X(t)\vert\le \sum_{i=1}^nX_i(t)=m(t)\le C(1+\vert x\vert).$$

\color{black}
\begin{Theorem}\label{thm:ver1} Consider the framework $(U1)(S2)$.
\begin{itemize}
    \item[$(i)$] The value function $V(x)$ is the unique classical solution of the HJB equation \eqref{hjbgen} in the set
$\mathcal C=\{v\in C^1(\mathbb R^n_+-\{0\})\cap C^0(\mathbb R^n_+) \}.$
\item[$(ii)$] Consequently,  $V(x)=w(x)\equiv A u(\langle x, \mathbf{e}\rangle)+B$, \  where $w(x)$ is given by    \eqref{eq:solhjb}, and the unique optimal feedback map  is   $$  U^*(x)=0,\ \  \forall i\not\in F, \quad U^*(x)=\theta^*\langle  x, \mathbf e\rangle,\ \  \forall i \in F.$$
\end{itemize} 
\end{Theorem}
\begin{proof} Let $v \in \mathcal C$ be any solution of \eqref{hjbgen}. Let  $x$ be any initial condition, $c\in \mathbb A$ any admissible control, and $X(t)$ the associated trajectory.  Using \eqref{eq:growth_X} and the fact that $v$ is locally bounded on $\mathbb R^n_+$, we have $|v(X(t))| \leq C_{x}$, for all $t \geq 0$. It follows
\begin{equation}\label{eq:behaviour_infty_vX}
    e^{-\rho t} v(X(t)) \xrightarrow{t \to \infty} 0.
\end{equation}
Since the fundamental identity applies and $H(p)\ge h(p)$ for every $p\in\mathbb R^n$, we have
$$v(x)\ge \int_0^te^{-\rho t}\sum_{i\in F}u(c_i(s))ds+e^{-\rho t}v(X(t)), \quad \forall t\ge0$$
so that, letting $t \to \infty$ and using \eqref{eq:behaviour_infty_vX},  we obtain
\begin{equation}
    v(x)\geq \int_0^\infty e^{-\rho t}\sum_{i=1}^f u(c_i(s))ds= J\left(c ; x\right)
\end{equation}
implying $v(x)\ge V(x).$
On the other hand, $H(\nabla v(x))-h(\nabla v(x), c)=0$ 
if and only if $c_i(x)=c_i^*(x)$ where $c_i^*$ is given by \eqref{optcontgen}. If we set $X^*(t)=X(t; x, c^*)$ then
\begin{align}\label{eq:v=int_+e_rho_t_vX*}
    v(x)= \int_0^te^{-\rho t}\sum_{i=1}^f  u(c^*_i(s))ds+e^{-\rho t}v(X^*(t)),\end{align}
and letting $t \to \infty$  we have 
\begin{align}
    v(x)=\int_0^\infty e^{-\rho t}\sum_{i=1}^f  u(c^*_i(s))ds= J\left(c^* ; x\right).
\end{align}
This implies that $c^*$ is  optimal  and $v(x)=V(x)$. In particular, we have uniqueness of the classical solution to the HJB equation in the set $\mathcal C$. Since $w(x)$ defined by \eqref{eq:solhjb} is a solution in $\mathcal C$, then $v(x)=V(x)=w(x)$,  and $c^*$ is an optimal feedback control.

Moreover, $c^*$ is the unique optimal  control. Indeed, let $(\tilde c, \tilde X)$ be another optimal couple, where $\tilde X(t)=X(t;x,\tilde c)$. Then 
$V(x)=J(\tilde c; x).$ Now choose choose $v=V, X=\tilde X, c=\tilde c$ in the fundamental identity, let $t \to \infty$, and use  \eqref{eq:behaviour_infty_vX}. We derive
$$\int_0^\infty\big[H( \nabla V(\tilde X(s)))-h(\tilde c(s), \nabla V(\tilde X(s)))\big]ds=0,$$
and,  since $H(p) \geq h(p,c)$, this implies
   $$H( \nabla V(\tilde X(s)))-h(\tilde c(s), \nabla V(\tilde X(s)))=0, \quad  {\textit{for a.e. }} s \geq 0.$$ But then  $(\tilde c,\tilde X)$ must satisfy \eqref{optcontgen} with $v=V$, so that $\tilde c$ has to coincide with the unique maximizer $c^*$. 

\end{proof}

Our target set of solution $v$ to the HJB equation \eqref{hjbgen} changes in frameworks (U1)(S1) and (U2)(S3). We will make the following assumptions on $v$:

\begin{itemize}
    \item [(V1)] For any admissible couple  $(X,c)$ such that $J(x,c)>-\infty$, it is 
\begin{equation}\label{eq:limsup}
 \limsup_{t \to \infty}  \ e^{-\rho t} v(X(t)) \geq 0,
\end{equation}    \item [(V2)] If $(X^*, c^*)$ is the admissible couple satisfying \eqref{optcontgen}, then
\begin{equation}\label{eq:lim_X*}
   \lim_{t \to \infty} e^{-\rho t} v(X^*(t)) = 0,
\end{equation}
\end{itemize}
\begin{Theorem}\label{thm:ver2}
    Consider the frameworks $(U1)(S1)$ and $(U2)(S3)$.

    \begin{itemize}
        
    \item[$(i)$] The value function $V(x)$ is the unique classical solution of the HJB equation \eqref{hjbgen} in the set
 $$\tilde{\mathcal C}
 =\{v\in C^1(\mathbb R^n_+-\{0\}): (V1)(V2)\ \textrm{hold}\}.$$
\item [$(ii)$]  Consequently,  $V(x)=w(x)\equiv A u(\langle x, \mathbf{e}\rangle)+B$, \  where $w(x)$ is given by    \eqref{eq:solhjb}, and the unique optimal feedback map  is   $$  \psi^*(x)=0,\ \  \forall i\not\in F, \quad \psi^*(x)=\theta^*\langle  x, \mathbf e\rangle,\ \  \forall i \in F.$$
\end{itemize}
 \end{Theorem}

\begin{proof}
Let $v$ be any classical solution of the HJB equation in $\tilde{\mathcal C}$. By proceeding as in the proof of Theorem \ref{thm:ver1}, but using \eqref{eq:limsup}, \eqref{eq:lim_X*} in place of \eqref{eq:behaviour_infty_vX}, we get that $v=V$ and that the control $c^*=\psi(X^*(\cdot))$ and the associated trajectory $X^*$ as defined in \eqref{eq:optccl}, are optimal. Notice that if  $J(x,c)=-\infty$, we always have the trivial inequality $v(x)\geq  -\infty=J(x,c)$.

Then, to prove that the value function can be expressed as the solution $w(x)$ of the HJB equation \eqref{eq:solhjb}, it is enough to show that $w\in\tilde{\mathcal C}$. Clearly (V2) holds as, for $m^*$ given by Lemma \ref{lem:m^*} with $\theta=\theta^*$, one has
$$w(X^*(t))=Au(m^*(t))+B \xrightarrow{t \to \infty} Au(\hat m)+B.$$
Next we prove that $w(x)$ satisfies  (V1)  in the case of framework (U1)(S1) (the case (U2)(S3) can be treated similarly). We consider an initial state $x$ an an admissible control ${c}(t)$, with associated state ${X}(t)$ and mass $m(t)=\langle X(t), \mathbf e\rangle$,  such that $J(x, c)>-\infty$, and prove \begin{equation}\label{eq:liminf}
 \limsup_{t \to \infty}  \ e^{-\rho t}\left(\frac A{1-\sigma}m(t)+B\right)\geq 0
 \iff \liminf_{t \to \infty}   e^{-\rho t}   m(t)\le 0.
\end{equation}
with $\sigma>1$ and $A,B$ as in   Table \ref{TabOC}.  
We start by rewriting $J(x,c)$ as 
\begin{align*}
  J(x,&c)=
 -\frac 1{\sigma-1}\sum_{T=0}^\infty \int_{T}^{T + 1} e^{-\rho t} \sum_{i\in F} ({c}_i(t))^{1 - \sigma} dt \le -\frac{{f^\sigma}}{\sigma-1}\sum_{T=0}^\infty \int_{T}^{T + 1} \kern-7pt e^{-\rho t} \left(\sum_{i\in F} {c}_i(t)\right)^{1 - \sigma} \kern-10pt dt\\
 &
 \le-\frac{{f^\sigma}}{\sigma-1}\sum_{T=0}^\infty \delta_T\left(\int_{T}^{T + 1} \kern-10pte^{-\rho t} \delta_T^{-1}\sum_{i\in F} {c}_i(t) dt\right)^{1 - \sigma}=-\frac{{f^\sigma}}{\sigma-1}\sum_{T=0}^\infty \delta_T\left(\int_{T}^{T + 1} \kern-10pt e^{-\rho t} \delta_T^{-1}\sum_{i\in F} {c}_i(t) dt\right)^{1 - \sigma}
\end{align*}
where the first inequality follows by Jensen's inequality and in the last one we have again used Jensen's inequality with  
unitary measure on $[T,T+1]$ of density
$f(t)=e^{-\rho t}{\delta_T}^{-1}$ with $\delta_T=\int_T^{T+1}e^{-\rho t}dt=\frac{e^\rho-1}{\rho e^\rho} e^{-\rho T}$.

Now, integrating the equation for $m(t)$  in \eqref{stateqgen} over   $[T, T+1]$, and using the fact that $ m(t)\ge0$ and $\varphi(m(t))\le \Gamma$, we derive
\begin{align*}
  \int_{T}^{T+1}   \langle {c}(t), \mathbf{e} \rangle& dt\le\int_{T}^{T+1} \kern-10pt\varphi(m(t))m(t)dt-m(T+1)+m(T)\\
  &\le m(T)+\Gamma\int_{T}^{T+1}  m(t)dt\le  m(T)+(e^\Gamma-1)m(T)=e^\Gamma m(T),
\end{align*}
where in the last inequality we used $m'(t)\le\Gamma m(t)$ and Gronwall's Lemma on  $[T, T+1]$. Combining the two chains of inequalities above, and recalling that $c_i=0$ for all $i\not\in F$,  and that $x\mapsto x^{1-\sigma}$ with $\sigma>1 $ is decreasing, we obtain
$$J(x,c)\le -\frac{{f^\sigma} }{\sigma-1}\sum_{T=0}^\infty  \delta_T^\sigma\left( e^{-\rho T}\int_{T}^{T + 1}  \langle c(t), \mathbf e\rangle dt\right)^{1 - \sigma}\le -\frac{{f^\sigma}  }{\sigma-1}\sum_{T=0}^\infty  \delta_T^\sigma \left(e^{-\rho T}e^\Gamma m(T)\right)^{1 - \sigma} $$
 $$= -\frac{{e^{\Gamma(1-\sigma)} } }{\sigma-1}\left(\frac{f(e^\rho-1)}{\rho e^\rho}\right)^\sigma\sum_{T=0}^\infty   e^{-\rho T}\left(m(T)\right)^{1 - \sigma}.$$
Note that the series above needs to be finite (and positive) as $J(x,c)>-\infty$. Thus, necessarily
$$\lim_{T\to\infty}e^{-\rho T}\left(m(T)\right)^{1 - \sigma}=0,$$
and the proof of (V1) is complete.

Lastly, we prove that $c^*$ is the unique optimal control. Indeed, let $(\tilde c, \tilde X)$ be another optimal couple, where $\tilde X(t)=X(t;x,\tilde c)$. Then 
$V(x)=J(\tilde c; x).$ Now choose  $v=V, X=\tilde X, c=\tilde c$ in the fundamental identity, take $\limsup_{t \to \infty}$, and use  \eqref{eq:limsup}, then
$$\int_0^\infty\big[H( \nabla V(\tilde X(s)))-h(\tilde c(s), \nabla V(\tilde X(s)))\big]ds\leq 0.$$
Thus, we must have $\tilde c=c^*$.
\end{proof}

 \subsection{Other Proofs}
\label{subapp:otherproofs}

\begin{proof}[\textbf{Proof of Lemma \ref{lem:eigenBB}}]
           The matrix $E$ can be written also as the product of the  two vector matrices $\xi=\sum_{i\in F}e_i$ (where $\{e_i\}$ is the canonical base in $\mathbb R^n$) and $\mathbf e^\top$, namely 
 $E=\xi \mathbf{e}^\top.$
Moreover, the range of $E^\top$ is $\mathrm{Im}E=\{\alpha\mathbf{e}\, :\, \alpha\in \mathbb R\}$, $\mathbf{e}$ is an eigenvector of $E^\top$ with eigenvalue $\langle \mathbf e, \xi\rangle=f$. Thus $\mathbf{e}$ is an eigenvector of $D+B-\theta E^\top+f\theta I$ with eigenvalue 0. This implies the existence of an eigenvector $\zeta_\theta$ of the transpose matrix associated to the same eigenvalue 0, and $(i)$ is proved. 

 To prove $(ii)$ we first show that any generalized eigenvector $v$ of $D+B-\theta E^\top+f\theta I$ with an  eigenvalue $\lambda_i \neq0$ is orthogonal to $\mathbf e$. Thus, consider an eigenvalue $\lambda_i \neq 0$ ($i\geq 2$) and $v_i$ as an element of the generalized eigenspace $V(\lambda_{i})$. Hence, there exist a strictly positive integer $m$ such that $(D+B-\theta E^\top+f\theta I-\lambda_iI)^{m}v_i  = {0}$. Therefore,
 $$0=\mathbf e^\top\left [(D+B-\theta E^\top+f\theta I-\lambda_iI)^m v_i \right ]=  (0-\lambda_i)^m\mathbf e^\top \;v_i,$$ 	    and given that $\lambda_i\not=0$, then $\mathbf e^\top\;v_i=0$, implying $\eta$ is orthogonal to $v_i$.\\ Thus, since $Ev_i=\xi\langle\mathbf e, v_i\rangle=0$,  $v_i$ is also a generalized eigenvector for $D+B-\theta E^\top+f\theta I$ associated with the eigenvalue $\lambda_i$, implying the second assertion.
\end{proof}

\begin{proof}[\textbf{Proof of Proposition \ref{pr:lrs}}]
   
(i) As discussed immediately before the proposition, we know that for $0<\theta<\min\{\theta_1,\theta_2\}$
the matrix $M:=D+B^\top-\theta E+f\theta I$ has a simple eigenvalue $0$ with left eigenvector $e$
and right eigenvector $\zeta_\theta\gg0$, and all other eigenvalues have strictly negative real part.
Hence the direction of the trajectory converges to the (unique) dominant eigenvector:
$Y(t)\to \zeta_\theta/\langle e,\zeta_\theta\rangle$, and $\zeta_\theta$ is strictly positive.

(ii) We look at $Y(t)$. Note first that if $Y(0)=\zeta_\theta/\langle e,\zeta_\theta\rangle$, then $Y(t)\equiv\zeta_\theta/\langle e,\zeta_\theta\rangle$ for all $t\ge0$.

Let $S:=\{y\in\mathbb{R}^n:\langle e,y\rangle=1\}$ and set $\tilde\zeta:=\zeta_\theta/\langle e,\zeta_\theta\rangle$.
Since $M$ has a simple eigenvalue $0$ with left eigenvector $e$ and right eigenvector $\zeta_\theta\gg0$, 
while all other eigenvalues have strictly negative real part, there exist $C,\eta>0$ such that,
writing $P$ for the projector onto the direction of $\zeta_\theta$, we have
\[
\|e^{Mt}-P\|\le C e^{-\eta t}\qquad\text{for all }t\ge0.
\]
Hence $e^{Mt}y\to\tilde\zeta$ uniformly on compact subsets of $S$; taking a small closed ball 
$K_L:=\{y\in S:\|y-\tilde\zeta\|\le L\}$ one obtains $e^{Mt}y\in\mathbb{R}^n_{++}$ for all $t\ge0$ and all $y\in K_L$,
which is exactly the cone condition stated for $x$ (via $Y(0)=x/\langle x,e\rangle$). The claim follows.
\end{proof} 
\begin{proof}[\textbf{Proof of Theorem \ref{th:ME}}]
Regarding $(i)$, the Lispschitz-continuity is obvious since $\psi^*$ is linear. For the boundary inequality we proceed as follows. If $i$ is not in $F$ then $\psi^*_i=0$ and there the inequality is verified.
So fix $i\in F$ and $x\in\mathbb R_+^n$ with $x_i=0$. Denote by $\underline b$ the $min_{i\neq j} b_{ij}$, 
\[
\big\langle (D+B^\top)x, e_i\big\rangle
=\sum_{j\neq i} b_{ji}x_j - \Big(\sum_{j\neq i} b_{ij}\Big)x_i
=\sum_{j\neq i} b_{ji}x_j
\;\ge\; \underline b \sum_{j\neq i} x_j .
\]
Since $x_i=0$, $\sum_{j\neq i} x_j=\langle \mathbf e,x\rangle$. Therefore,
\[
\big\langle (D+B^\top)x, e_i\big\rangle
\;\ge\; \underline b\,\langle \mathbf e,x\rangle
\;>\; \hat\theta\,\langle \mathbf e,x\rangle
\;=\; \psi_i(x),
\]
where the strict inequality is implied by $0<\hat\theta<\underline b$ since $\langle \mathbf e,x\rangle>0$.

Now we prove $(ii)$ and $(iii)$.

Fix $x\in\mathbb R_+^n$. To show the claim, it is enough to  suppose all players $j\in F$ other than $i$ use the affine stationary feedback $\psi_j(x)=\hat\theta\,\langle \mathbf e,x\rangle=\hat\theta\, m$, where $\hat\theta$ is as in Table \ref{Tabellone},  to show that player $i$'s best response is $\psi_i(x)=\hat\theta m$, and that the associated value function has the form $V^i(x)=A\,u(m)+B$, with the corresponding constants $A,B$ in Table \ref{Tabellone}. 

\medskip
Let us consider the problem of player $i$ against $\psi_{-i}$, and denote by $v^i$ the value function of player $i$. The HJB reads
\begin{multline}\label{eq:HJB-i}
\rho\,v^i(x)=\max_{c_i\ge 0}\Big\{u(c_i)-c_i\,\partial_{x_i}v^i(x)\Big\}
+\big\langle \nabla v^i(x),(D+B^\top)x+\varphi(\langle\mathbf e,x\rangle)\,x\big\rangle
\\ -\sum_{j\in F\setminus\{i\}}\partial_{x_j}v^i(x)\,\hat\theta\, \langle \langle\mathbf e,x\rangle.
\end{multline}
We verify that, for a suitable choice of the constants $A$ and $B$, the function
\[
v^i(x)=A\,u(\langle\mathbf e,x\rangle)+B =A\,u(m)+B
\]
is a solution of such an equation. Indeed, we have
\[
v^i(x)=A\,u'(m)\,\mathbf e,\quad p_i:=\partial_{x_i}v^i=A\,u'(m).
\]
So the above HJB equation \eqref{eq:HJB-i} is rewritten as
\begin{equation}\label{eq:HJB-verify}
\rho\,v^i(x)=\max_{c_i\ge0}\{u(c_i)-c_i\,p_i\}+\big\langle\nabla v^i,(D+B^\top)x+\varphi(m)\,x\big\rangle
-\sum_{j\in F\setminus\{i\}}\partial_{x_j}v^i\,\hat\theta\,m .
\end{equation}
Since $(D+B)\mathbf e=\mathbf 0$, we have $\langle\nabla v^i,(D+B^\top)x\rangle=0$, and
\[
\big\langle\nabla v^i,\varphi(m)\,x\big\rangle=A\,u'(m)\,\varphi(m)\,m, \qquad
\sum_{j\ne i}\partial_{x_j}v^i\,\hat\theta\,m=(f-1)\,A\,u'(m)\,\hat\theta\,m .
\]
All in all, in each case we need to prove that
\begin{equation}\label{eq:HJB-identity}
  \rho\big[A\,u(m)+B\big]
  \;=\;
  H\!\big(A\,u'(m)\big)
  \;+\; A\,u'(m)\,\varphi(m)\,m
  \;-\;(f-1)\,A\,u'(m)\,\hat\theta\,m,
\end{equation}
where $H(p)=\sup_{c\ge0}\{u(c)-cp\}$. We remark that the maximizer is $c_i^*=(u')^{-1}(p_i)$.

We now verify, case by case, that:
\begin{itemize}
\item[(a)] $c_i^*(x)=\hat\theta\,m$ (best response exactly as claimed);
\item[(b)] the identity \eqref{eq:HJB-identity} holds with $v^i(x)=A\,u(m)+B$,
\end{itemize}
using the given $(\hat\theta,A,B)$ from Table \ref{Tabellone}.

\smallskip
\noindent \textbf{Case (U1)-(S1)}: $u(c)=\dfrac{c^{1-\sigma}}{1-\sigma}$, $\sigma\neq1$ and $\varphi(m)=\Gamma\!\left(1-\frac1K m^{\sigma-1}\right)$.

Here $u'(c)=c^{-\sigma}$, $(u')^{-1}(p)=p^{-1/\sigma}$. Then
\[
c_i^*=(A\,m^{-\sigma})^{-1/\sigma}=A^{-1/\sigma}m.
\]
With the prescribed $A=\hat\theta^{-\sigma}$ we get \emph{(a)} $c_i^*=\hat\theta\,m$.
The maximized Hamiltonian equals $H(p_i)=\frac{\sigma}{1-\sigma}p_i^{1-1/\sigma}=\frac{\sigma}{1-\sigma}A^{1-\frac1\sigma}m^{1-\sigma}
=\frac{\sigma}{1-\sigma}\hat\theta^{1-\sigma}m^{1-\sigma}$.
Moreover
\[
A\,u'(m)\,\varphi(m)\,m
=A\,m^{-\sigma}\Big[\Gamma m-\frac{\Gamma}{K}m^\sigma\Big]
=A\,\Gamma\,m^{1-\sigma}-\frac{A\Gamma}{K},
\]
and
\[
(f-1)\,A\,u'(m)\,\hat\theta m=(f-1)\,A\,\hat\theta\,m^{1-\sigma}.
\]
The HJB \eqref{eq:HJB-verify} reduces to two scalar identities:
\[
\ \frac{\rho A}{1-\sigma}=\frac{\sigma}{1-\sigma}\hat\theta^{1-\sigma}+A\Gamma-(f-1)A\hat\theta\quad\text{(coeff.\ of $m^{1-\sigma}$),}
\]
and
\[
\ \rho B=-\frac{A\Gamma}{K}\quad\text{(constant term).}
\]
Substituting $A=\hat\theta^{-\sigma}$ and $\hat\theta=\dfrac{\rho+\Gamma(\sigma-1)}{1+f(\sigma-1)}$ (Table \ref{Tabellone}) both equalities hold identically, hence \eqref{eq:HJB-verify} is satisfied and (b) follows.

\smallskip
\noindent 
\textbf{Case (U1) - (S2)}: $u(c)=\dfrac{c^{1-\sigma}}{1-\sigma}$, $\sigma\in(0,1)$; and $\varphi(m)=m^{\sigma-1}-\delta$.
Again $(u')^{-1}(p)=p^{-1/\sigma}$, hence with $A=\hat\theta^{-\sigma}$ we get \emph{(a)} $c_i^*=\hat\theta\,m$.
Here
\[
H(p_i)=\frac{\sigma}{1-\sigma}\hat\theta^{1-\sigma}m^{1-\sigma},
\]
\[A\,u'(m)\,\varphi(m)\,m=A\,[1-\delta\,m^{1-\sigma}],
\]
and
\[
(f-1)\,A\,u'(m)\,\hat\theta m=(f-1)\,A\,\hat\theta\,m^{1-\sigma}.
\]
Thus \eqref{eq:HJB-verify} is equivalent to the system of the two following equations:
\[
\ \frac{\rho A}{1-\sigma}=\frac{\sigma}{1-\sigma}\hat\theta^{1-\sigma}-\delta A-(f-1)A\hat\theta \quad\text{(from the coeff.\ of $m^{1-\sigma}$)},
\]
\[
\ \rho B=A \quad\text{(from the constant term).}
\]
Substituting $A=\hat\theta^{-\sigma}$ and $\hat\theta=\dfrac{\rho+\delta(1-\sigma)}{1-f(1-\sigma)}$ (Table \ref{Tabellone}) the two identities are verified, hence (b) holds.

\smallskip
\noindent\textbf{Case (U2) - (S3)}: $u(c)=\ln c$;  $\varphi(m)=\Gamma\!\left(1-\frac1K\ln m\right)$.

Here $u'(c)=1/c$, $(u')^{-1}(p)=1/p$, so $c_i^*=\dfrac{m}{A}=\hat\theta\,m$ with the prescribed $A=\hat\theta^{-1}$, giving \emph{(a)}.

Moreover
\[
H(p_i)=-\ln p_i-1=\ln m-\ln A-1,
\]
\[
A\,u'(m)\,\varphi(m)\,m=A\,\Gamma-\frac{A\Gamma}{K}\,\ln m,
\]
and
\[
\sum_{j\ne i}\partial_{x_j}v^i\,\hat\theta m=(f-1)\,A\,\hat\theta=(f-1).
\]
Matching the coefficients of $\ln m$ and the constants in \eqref{eq:HJB-verify} yields
\[
\ \rho A=1-\frac{A\Gamma}{K}\quad\text{(coeff.\ of $\ln m$),}
\qquad
\rho B=-\ln A-1+A\Gamma-(f-1)\quad\text{(constant term).}
\]
With $A=\hat\theta^{-1}$ and $\hat\theta=\rho+\Gamma/K$ (Table \ref{Tabellone}), both identities are satisfied, hence (b) holds.

\medskip
\noindent\emph{Verification and conclusion.}
In each case above, the candidate $v^i(x)=A\,u(\langle\mathbf e,x\rangle)+B$ with the corresponding $A,B$ solves \eqref{eq:HJB-i} with maximizer $c_i^*(x)=(u')^{-1}(\partial_{x_i}v^i(x))=\hat\theta\,\langle\mathbf e,x\rangle$. Arguing exactly as in Theorems \ref{thm:ver1} and \ref{thm:ver2} (applied with $(V1)$--$(V2)$ where needed) this $v^i$ coincides with the value function and $c_i^*$ is the unique optimal Markovian feedback for player $i$ against $\psi_{-i}$. Since $i$ was arbitrary and all players share the same primitives, the strategy profile
\[
\psi_i^*(x)=
\begin{cases}
\hat\theta\,\langle\mathbf e,x\rangle,& i\in F,\\
0,& i\notin F,
\end{cases}
\]
is a Markovian equilibrium, and all players share the same value $V^i(x)=A\,u(\langle\mathbf e,x\rangle)+B$. This proves items $(i)$--$(ii)$.
\end{proof}

\end{small}

\begin{small}
\paragraph{\textbf{Acknowledgments.}} Filippo de Feo acknowledges funding by the Deutsche Forschungsgemeinschaft (DFG, German Research Foundation) – CRC/TRR 388 "Rough Analysis, Stochastic Dynamics and Related Fields" – Project ID 516748464 and by INdAM (Instituto Nazionale di Alta Matematica F. Severi) - GNAMPA (Gruppo Nazionale per l'Analisi Matematica, la Probabilità e le loro Applicazioni).
\end{small}
\bigskip
\end{document}